\documentclass[twoside,11pt]{article}
\usepackage{amsmath}
\usepackage{amsthm}
\usepackage{thmtools}
\usepackage[hidelinks]{hyperref} 
\usepackage[nameinlink]{cleveref}
\usepackage[abbrvbib,nohyperref,preprint]{jmlr2e}

\usepackage{lastpage}
\jmlrheading{26}{2025}{1-\pageref{LastPage}}{12/24; Revised
11/25}{12/25}{24-2189}{ Martin Bauer, Facundo M\'{e}moli, Tom Needham, and Mao Nishino}
\ShortHeadings{The Z-Gromov-Wasserstein Distance}{Bauer, M\'{e}moli, Needham, and Nishino}

\usepackage[normalem]{ulem}
\usepackage{xcolor}
\usepackage{framed}
\usepackage{amssymb}
\usepackage{autonum}
\usepackage{mathrsfs}
\usepackage{mathabx}
\usepackage{relsize}
\usepackage{enumitem}
\usepackage{comment}
\usepackage{wrapfig}
\usepackage{booktabs}
\usepackage{mathtools}
\usepackage{cancel}   
\usepackage{url}
\usepackage[utf8]{inputenc}
\usepackage{booktabs}
\usepackage{tabularx}
\usepackage{array}
\usepackage{needspace}

\definecolor{darkblue}{rgb}{0.0, 0.0, 0.8}

\theoremstyle{definition}
\newtheorem{examples}[theorem]{Examples}
\newtheorem{question}[theorem]{Question}

\newcommand{\M}{\mathcal{M}}
\newcommand{\MZp}{\M_{\sim}^{Z,p}}
\newcommand{\dgwz}{\mathrm{GW}_p^Z}
\newcommand{\dgwinf}{\mathrm{GW}_{\infty}^Z}

\newcommand{\dgwRn}{\mathrm{GW}_p^{\mathbb{R}^n}}
\newcommand{\dgw}{\mathrm{GW}_p}
\newcommand{\disp}{\mathrm{dis}_p}
\newcommand{\eccout}{\textnormal{ecc}_{p,X,Y}^{\textnormal{out}}}
\newcommand{\eccoutpX}{\textnormal{ecc}_{p,z_0,X}^{\textnormal{out}}}
\newcommand{\eccoutpY}{\textnormal{ecc}_{p,z_0,Y}^{\textnormal{out}}}
\newcommand{\eccin}{\textnormal{ecc}_{p, X,Y}^{\textnormal{in}}}
\newcommand{\eccinpX}{\textnormal{ecc}_{p,z_0,X}^{\textnormal{in}}}
\newcommand{\size}{\textnormal{size}_{p,z_0}}
\newcommand{\R}{\mathbb{R}}
\newcommand{\define}[1]{\textbf{#1}}

\firstpageno{1}

\begin{document}

\title{The Z-Gromov-Wasserstein Distance}

\author{%
\name Martin Bauer \email mbauer2@fsu.edu\\[0.2em]
\name Tom Needham \email tneedham@fsu.edu\\[0.2em]
\name Mao Nishino \email mnishino@fsu.edu\\[0.2em]
\addr Department of Mathematics\\
Florida State University\\
Tallahassee, FL 32306-4510, USA
\AND
\name Facundo M\'{e}moli \email facundo.memoli@rutgers.edu\\[0.2em]
\addr Department of Mathematics\\
Rutgers University\\
Piscataway, NJ 08854-8019, USA
}

\editor{Qiang Liu}

\maketitle

\begin{abstract}
The Gromov-Wasserstein (GW) distance is a powerful tool for comparing metric measure spaces which has found broad applications in data science and machine learning. Driven by the need to analyze data sets whose objects have increasingly complex structure (such as node and edge-attributed graphs), several variants of GW distance have been introduced in the recent literature. With a view toward establishing a general framework for the theory of GW-like distances, this paper considers a vast generalization of the notion of a metric measure space: for an arbitrary metric space $Z$, we define a $Z$-network to be a measure space endowed with a kernel valued in $Z$.  We introduce a method for comparing $Z$-networks by defining a generalization of GW distance, which we refer to as $Z$-Gromov-Wasserstein ($Z$-GW) distance. This construction subsumes many previously known metrics and offers a unified approach to understanding their shared properties. This paper demonstrates that the $Z$-GW distance defines a metric on the space of $Z$-networks which retains desirable properties of $Z$, such as separability, completeness, and geodesicity. Many of these properties were unknown for existing variants of GW distance that fall under our framework. Our focus is on foundational theory, but our results also include computable lower bounds and approximations of the distance which will be useful for practical applications.
\end{abstract}


\newpage
\tableofcontents

\section{Introduction}

Frequently in pure and applied mathematics, one requires a mechanism for measuring the dissimilarity of objects which are a priori incomparable. For example, the Gromov-Hausdorff distance provides such a mechanism for comparing abstract Riemannian manifolds and famously provides a notion of convergence for sequences of manifolds, under which certain geometric features are preserved~\citep{cheeger1997structure,Gromov2006Metric}. The Gromov-Hausdorff distance, in fact, defines a metric on the space of arbitrary compact metric spaces. This is useful in applied areas such as shape analysis and data science, where it is increasingly common to encounter analysis tasks involving ensembles of complex data objects which naturally carry metric structures, such as point clouds or graphs---the ability to compute distances between these objects opens these problems to metric-based techniques \citep{memoli2004comparing,memoli2005theoretical,chazal2009gromov,JMLR:v11:carlsson10a, memoli2012some}. Introducing measures on the metric spaces allows for the application of ideas from the field of optimal transport, leading to the following construction, introduced and studied in~\citep{memoli2007,memoli2011}: given a pair of metric spaces endowed with probability measures $(X,d_X,\mu_X)$ and $(Y,d_Y,\mu_Y)$, the \emph{Gromov-Wasserstein $p$-distance} (for $p \geq 1$) between them is given by the quantity 
\begin{equation}\label{eqn:GW_intro}
\mathrm{GW}_p(X,Y) = \frac{1}{2} \inf_{\pi} \left( \iint_{(X\times Y)^2} |d_X(x,x') - d_Y(y,y')|^p \pi(dx \times dy) \pi(dx' \times dy') \right)^{1/p},
\end{equation}
where the infimum is over probability measures $\pi$ on $X \times Y$ whose left and right marginals are $\mu_X$ and $\mu_Y$, respectively. Intuitively, such a measure describes a probabilistic correspondence between the sets $X$ and $Y$, and the integral in~\eqref{eqn:GW_intro} measures the extent to which such a correspondence distorts the metric structures of the spaces. The optimization problem \eqref{eqn:GW_intro} therefore seeks a probabilistic correspondence which minimizes total metric distortion. It was shown by~\citet{memoli2007} that the Gromov-Wasserstein $p$-distance defines a metric on the space of all triples $(X,d_X,\mu_X)$, considered up to measure-preserving isometries, when $(X,d_X)$ is compact and $\mu_X$ is fully supported (the compactness assumption can be relaxed; see \citealt{memoli2022comparison,sturm2023space}).

The Gromov-Wasserstein (GW) framework has become a popular tool in data science and machine learning---see~\citep{demetci2022scot,chowdhury2021generalized,
chapel2020partial,
xu2020gromov,
chowdhury2020gromov,
alvarez2018gromov}, among many others---due to its flexibility in handling diverse data types, its  robustness to changes in metric or measure structures~\citep[Theorem 5.1]{memoli2011}, and recent advances in scalable and empirically accurate computational schemes for its estimation~\citep{peyre2016gromov,xu2019scalable,chowdhury2021quantized,scetbon2022linear,li2023efficient,vedula2024scalable}. It was observed by~\citet{peyre2016gromov} that the formula for GW distance in~\eqref{eqn:GW_intro} still makes sense when the condition that $d_X$ and $d_Y$ are metrics is relaxed; that is, the formula gives a meaningful comparison between structures of the form $(X,\omega_X,\mu_X)$, where $\omega_X:X \times X \to \R$ is an arbitrary measurable function.  This point of view was later studied formally by \citet{chowdhury2019gromov}, where it was shown that the GW distance defines a metric on the space of such triples $(X,\omega_X,\mu_X)$, considered up to a natural notion of equivalence. This is convenient, for example, when handling graph data sets, where it is natural to represent a graph's structure through its adjacency function, Laplacian, or heat kernel.

Driven by applications to machine learning on increasingly complex data types, several variants of GW distance have been introduced in the literature~\citep{memoli2009spectral,vayer2020fused,woojin-thesis,memoli2021ultrametric,arya2023gromov,yang2023exploiting}. For example, \citet{vayer2020fused} introduces an adaptation of GW distance which is equipped to handle graphs endowed with \emph{node features}; that is, graphs whose nodes are endowed with values in some auxiliary metric space (in fact, this idea goes back further to, at least,~\citealt[Section 5]{chazal2009gromov}, where a similar construction provided variants of both the Gromov-Hausdorff distance and the $p=\infty$ version of GW distance). Each time a new variant of the GW distance is introduced in the literature, its metric properties have been re-established. Reviewing these proofs reveals that they tend to follow a common template, suggesting the existence of a higher-level explanation of their shared properties. 

\begin{figure}
    \centering
    \includegraphics[width=0.85\linewidth]{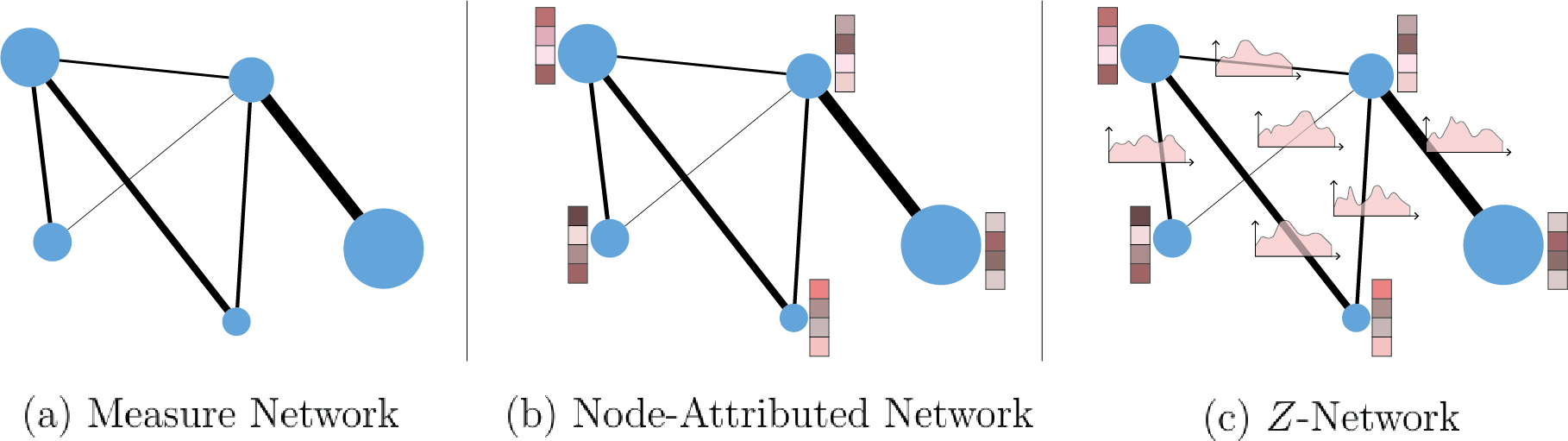}
    \caption{Schematic illustration of types of networks. { (a)} A graph with edge weights and node weights (each visualized by size variations). This structure is encoded as a {measure network}, and two such structures can be compared through the {Gromov-Wasserstein distance} \citep{memoli2007,chowdhury2019gromov}. { (b)} A weighted graph with additional node features, consisting of an assignment of a vector in $\R^n$ to each node (visualized as a column vector). These objects can be compared via the {Fused Gromov-Wasserstein distance} \citep{titouan2019optimal,vayer2020fused}. { (c)} Additionally, a graph can be endowed with edge features, assigning a point in some fixed metric space $Z$ to each edge (here, we visualize a 1-dimensional probability distribution attached to each edge). These complex objects are modeled as {$Z$-networks}, in the language of this paper, and two such objects can be compared through our proposed framework. By choosing an appropriate target space $Z$, one recovers many notions of distance between structured objects that have appeared previously in the literature---see \Cref{tab:metrics}.}
    \label{fig:Z-networkSchematic}
\end{figure}

The main goal of this paper is to formally develop a high-level structure which encompasses the GW variants described in the previous paragraph. This is accomplished by studying triples $(X,\omega_X,\mu_X)$ where $\omega_X:X \times X \to Z$ is now a function valued in some fixed but arbitrary metric space $(Z,d_Z)$; we call such a triple a \emph{$Z$-network} (see \Cref{fig:Z-networkSchematic}). Two $Z$-networks can be compared via a natural generalization of the GW distance~\eqref{eqn:GW_intro}: the \emph{$Z$-Gromov-Wasserstein $p$-distance} between $Z$-networks  $(X,\omega_X,\mu_X)$ and $(Y,\omega_Y,\mu_Y)$ is
\begin{equation}\label{eqn:ZGW_intro}
\mathrm{GW}_p^Z(X,Y) = \frac{1}{2}\inf_{\pi} \left( \iint_{(X\times Y)^2} d_Z\big(\omega_X(x,x'), \omega_Y(y,y')\big)^p \pi(dx \times dy) \pi(dx' \times dy') \right)^{1/p}.
\end{equation}
One should immediately observe that if $Z = \R$ (with its standard metric) then $\mathrm{GW}_p^Z$ recovers $\mathrm{GW}_p(X,Y)$. Moreover, we will show that the GW variants described above correspond to $Z$-GW distances for other appropriate choices of $Z$. The broad goal of the paper is to establish a general theoretical framework for GW-like distances, as understanding geometric properties of this general structure eliminates the need to re-derive the properties for GW variants which fall into our framework. Besides enabling the avoidance of such redundancies for future variants of GW distance, this general perspective leads to novel insights about existing metrics: several of our theoretical results were previously unknown or only shown in weaker forms for distances already studied in the literature.

\subsection{Main Contributions and Outline}

Let us now outline our main results, which are stated here somewhat informally.

\begin{itemize}[leftmargin=*]
    \item We show that several GW-like metrics appearing in the literature can be realized as $Z$-GW distances. In particular, \Cref{thm:ZGW_generalizes} states that the Wasserstein distance, (standard) GW distance~\citep{memoli2007}, ultrametric GW distance~\citep{memoli2021ultrametric}, $(p,q)$-GW distance~\citep{arya2023gromov}, Fused GW distance~\citep{vayer2020fused}, Fused Network GW distance~\citep{yang2023exploiting}, spectral GW distance~\citep{memoli2009spectral} and GW distance between weighted dynamic metric spaces~\citep{woojin-thesis} can all be realized as $Z$-GW distances. 
    We also explain how the graphon cut metric~\citep{borgs2008convergent} fits into our framework. Finally, we show that $Z$-GW distances define natural metrics on spaces of shape graphs~\citep{guo2022statistical}, connection graphs~\citep{robertson2023generalization} and probabilistic metric spaces~\citep{menger1942statistical}. These results are summarized in \Cref{tab:metrics}.
    \item \Cref{thm:ZGW_is_a_metric} says that, when $(Z,d_Z)$ is separable, the $Z$-GW distance $\mathrm{GW}_p^Z$ defines a metric on the space $\MZp$ consisting of $Z$-networks, considered up to a natural notion of equivalence (\Cref{def:weak_isomorphism}). As a consequence, \Cref{cor:fusedGW_triangle} shows that Fused GW and Fused Network GW distances are metrics; these were previously only shown to satisfy a certain \emph{weak triangle inequality}.
    The proof of \Cref{thm:ZGW_is_a_metric} relies on a technical result, which says that the solution of the optimization problem~\eqref{eqn:ZGW_intro} is always realized (\Cref{thm: optimal coupling}).
    \item Several geometric and topological properties of the metric space $(\MZp, \mathrm{GW}_p^Z)$ are established: it is separable (\Cref{prop: separability}), complete if and only if $Z$ is (\Cref{thm:complete}), contractible, regardless of the topology of $Z$ (\Cref{thm:contractible}), and geodesic if $Z$ is (\Cref{thm:geodesic}). Many of these results are novel when restricted to the examples of $Z$-GW distances covered in \Cref{thm:ZGW_generalizes}.
    \item Approximations of $Z$-GW distances are established via certain polynomial-time computable lower bounds in \Cref{thm:lower_bounds}. Moreover, \Cref{thm: approximation_Rn} provides a quantitative approximation of general $Z$-GW distances by $\R^n$-GW distances, the latter of which should be efficiently estimable through mild adaptations of existing GW algorithms. 
\end{itemize}
\begin{table}[t]
\centering
\small
\setlength{\tabcolsep}{3pt}
\renewcommand{\arraystretch}{1.0}

\begin{tabularx}{\linewidth}{@{}
  >{\raggedright\arraybackslash}p{0.22\linewidth}
  >{\raggedright\arraybackslash}p{0.20\linewidth}
  >{\raggedright\arraybackslash}X
  >{\raggedright\arraybackslash}p{0.22\linewidth}
@{}}
\toprule
Metric & Data Type & Target space $Z$ & References\\
\midrule
Wasserstein Distance & Distributions over $(Z,d_Z)$ & $Z$, $d_Z$ (any metric space) & \citet{villani2003topics} \\
Gromov-Wasserstein (GW) Distance & Metric measure spaces & $\mathbb{R}$, standard metric & \citet{memoli2007,memoli2011,chowdhury2019gromov} \\
Ultrametric GW Distance & Ultrametric measure spaces & $\mathbb{R}$, max metric & \citet{memoli2021ultrametric} \\
$(p,q)$-GW Distance & Metric measure spaces & $\mathbb{R}$, $\Lambda_q$ (see Equation \ref{eqn:q_metrics}) & \citet{arya2023gromov} \\
Fused GW Distance & Node-attributed graphs & $\mathbb{R} \times Y$ for a metric space $Y$ & \citet{titouan2019optimal,vayer2020fused} \\
Fused Network GW Distance & Node- and edge-attributed graphs & Product of node and edge spaces & \citet{yang2023exploiting,kawano2024multi} \\
Spectral GW Distance & Riemannian manifolds & Continuous functions, sup metric & \citet{memoli2009spectral} \\
GW Distance for Weighted DMSs & Weighted Dynamic Metric Spaces & Cts. functions, interleaving distance & \citet{woojin-thesis} \\
\addlinespace
\midrule
Shape Graph Distance & Embedded 1-d stratified spaces & Space of curves, any metric & \citet{sukurdeep2022new,srivastava2020advances,bal2024statistical} \\ 
Connection Graph Distance & $\mathrm{O}(n)$-attributed graphs & $\mathrm{O}(n)$,  Frobenius distance & \citet{bhamre2015orthogonal,singer2011angular,robertson2023generalization} \\
Probabilistic Metric Space Distance & Probabilistic metric spaces & Wasserstein space over $\mathbb{R}$ & \citet{menger1942statistical,wald1943statistical,kramosil1975fuzzy} \\
\bottomrule
\end{tabularx}
\caption{Summary of Gromov-Wasserstein-like distances which fall under the framework described in this paper. In each case, we provide the name of the metric, the type of data objects which it is able to compare, the choice $(Z,d_Z)$ which realizes the metric in our framework, and references to where the metric was first studied. The last three distances are novel to this paper, but handle data objects which have been studied in prior work.}
\label{tab:metrics}
\end{table}

The structure of the paper is as follows. After concluding the introduction section with a discussion of related work, the main definitions of the objects of interest are given in \Cref{sec:basic_definitions}. \Cref{sec:examples} describes examples of $Z$-GW distances, including some distances which have already appeared in the literature and some which are novel to the present paper. The basic geometric properties of $Z$-GW distances are established in \Cref{sec: metric_properties}. Finally, our results on approximations of $Z$-GW distances are presented in \Cref{sec:approximations}. The main body of the paper concludes with a discussion of open questions and future research directions in~\Cref{sec:discussion}. Proofs of our main results are included in the main body of the paper, but we relegate proofs of a few technical results to \Cref{sec:appendix}.

\subsection{Related Work}\label{sec:related_work}
 Concepts related to the Z-Gromov-Wasserstein distance have been considered previously, see, e.g., \citep{jain2009structure,peyre2016gromov,yang2023exploiting,kawano2024multi}. We now give precise comparisons of these previous works to the setting of the present article.

 \citet{peyre2016gromov} already considers variants of GW distance between $\R$-networks (i.e., objects of the form $(X,\omega_X,\mu_X)$ with $\omega_X$ valued in $\R$) where the integrand of~\eqref{eqn:GW_intro} is replaced with a more general \emph{loss function} of the form 
 \[
 X \times Y \times X \times Y \ni (x,y,x',y') \mapsto L(\omega_X(x,x'),\omega_Y(y,y')) \in \R,
 \]
 with $L: \R \times \R \to \R$ an arbitrary function. The article places a specific focus on the case where $L$ is the Kullback-Leibler divergence, which is not a metric, and therefore does not fall within our framework. The geometric properties that we establish for $Z$-GW distances in this paper largely depend on the assumption that $(Z,d_Z)$ is a metric space, so we restrict our attention to this setting.

    While we were in the process of preparing the article, \citet{yang2023exploiting} independently introduced the idea of a $Z$-network, although in the slightly less general context that network kernels are assumed to be bounded and continuous. A version of the $Z$-GW distance is also formulated therein, and their framework is applied to the analysis of attributed graph data. The work of~\citet{yang2023exploiting} is primarily geared toward computational applications, while the present paper is focused on developing the mathematical properties of the $Z$-GW distance; all of our results are novel, with connections to previous results clarified as necessary below. We also remark here that the work of Yang et al.\ gives a notion of $Z$-GW distance which appears at first glance to be more general, in that it includes additional terms handling node features and edge weights; we show, however, in \Cref{sec:attributed_graphs} that these terms are, in fact, not necessary. A similar framework to that of Yang et al.\ was also recently studied by~\citet{kawano2024multi}, where the focus was restricted to $\R^n$-valued kernels on finite sets, also with a view toward applications to attributed graph data sets. In summary, \cite{yang2023exploiting} and \cite{kawano2024multi} focus on extending the notion of fused Gromov-Wasserstein distance to do machine learning on attributed networks, whereas our goal is to develop a unifying theory on a more general class of GW-like distances.
    
    The first work to consider structures similar to $Z$-networks, to our knowledge, is that of \citet{jain2009structure}, where the objects under consideration are kernels over a finite set of fixed size $n$, valued in a fixed metric space. These structures are compared using a distance similar to the $Z$-GW distance, with a key difference being that the formulation does not involve measures: the optimization is performed over the conjugation action by the permutation group on $n$ letters, rather than on the space of couplings. Kernels on sets of different cardinalities are compared by ``padding" the smaller kernel with some fixed value to bring them to the same size---e.g., if the target metric space is a vector space, then the kernel is padded with zero vectors. The formulations of \citet{yang2023exploiting,kawano2024multi} and the present paper can then be viewed as OT-based extensions of the ideas of Jain and Obermayer, which are able to avoid the arguably unnatural padding operation.

    Another recent approach to a general framework for Gromov-Wasserstein-like distances is developed by~\citet{zhang2024geometry}. In that paper, the generalization is given by allowing optimization over spaces of measure couplings with a certain additional structure, rather than varying the target of the kernel function. That framework encompasses a different collection of variants of GW distances (namely, those of the form presented by~\citealt{chowdhury2023hypergraph,titouan2020co}) than those considered here. Moreover,~\citealt{zhang2024geometry} is mainly focused on curvature bounds and on a computational framework, rather than the structural results presented in this paper.

\section{\texorpdfstring{$Z$}{Z}-Gromov-Wasserstein Distances}\label{sec:basic_definitions}

We now present the main definitions and constructions of the paper.

\subsection{Background}

We begin by recalling some background terminology and notation on various notions of distance between probability measures. 

\subsubsection{Basic Terminology and Notation}\label{sec:basic_terminology}

We emphasize here that the term \define{metric space} is reserved for a pair $(Z,d_Z)$, where $d_Z:Z \times Z \to \R$ satisfies the usual axioms of symmetry, positive-definiteness and the triangle inequality. If $d_Z(z,z') = 0$ for some $z \neq z'$ (i.e., $d_Z$ does not satisfy the positive-definiteness axiom), then we refer to $d_Z$ as a \define{pseudometric} and $(Z,d_Z)$ as a \define{pseudometric space}. When the existence of a preferred (pseudo)metric is clear from context, we may abuse notation and use only $Z$ to denote the metric space. 

Recall that a topological space $Z$ is called a \define{Polish space} if it is completely metrizable and separable (i.e., contains a countable dense subset). When metrized by a particular complete metric $d_Z$, we call $(Z,d_Z)$ a \define{Polish metric space}. A standard assumption for many results in abstract measure theory is that the underlying space is Polish---for example, see \citep{srivastava2008course}.

Let $(X,\mu)$ be a measure space, $Y$ a measurable space and $\phi:X \to Y$ a measurable map. The \define{pushforward measure} on $Y$ is denoted $\phi_\ast \mu$. We recall that this is defined on a measurable subset $A \subset Y$ by $\phi_\ast \mu(A) = \mu(\phi^{-1}(A))$. 

Given a point $z \in Z$, the associated \define{Dirac measure} is the Borel measure $\delta_z$ on $Z$ defined on a Borel set $A \subset Z$ by  
\[
\delta_z(A) = \left\{\begin{array}{lr}
1 & \mbox{if $z \in A$}\\
0 & \mbox{if $z \not \in A$.}
\end{array}\right.
\]
Dirac measures will appear in several constructions below (e.g., \Cref{sec:probabilistic_metric_spaces}, the proofs of \Cref{prop:dense_subset} and \Cref{thm:complete}, etc.).

\subsubsection{Optimal Transport Distances}

Our constructions are largely inspired by concepts from the field of optimal transport (OT), which studies certain metrics on spaces of probability distributions, called \emph{Wasserstein distances}. The following concepts are well-established---see~\citep{villani2003topics} as a general reference on OT.

\begin{definition}[Coupling]
    Let $(X,\mu)$ and $(Y,\nu)$ be probability spaces. A \define{coupling} of $\mu$ and $\nu$ is a probability measure $\pi$ on $X \times Y$ with marginals $\mu$ and $\nu$, respectively. We denote the set of couplings between $\mu$ and $\nu$ as $\mathcal{C}(\mu,\nu)$.
\end{definition}

\begin{definition}[Wasserstein Distance]\label{def:wasserstein}
    Let $(Z,d_Z)$ be a Polish metric space and let $\mu$ and $\nu$ be Borel probability measures on $Z$ with finite $p$-th moments, for $p \in [1,\infty]$. The \define{Wasserstein $p$-distance} between $\mu$ and $\nu$ is
    \[
        \mathrm{W}_p(\mu,\nu) = \inf_{\pi \in \mathcal{C}(\mu,\nu)} \|d_Z\|_{L^p(\pi)}.
    \]
    This is written more explicitly, for $p < \infty$, as
    \[
    \mathrm{W}_p(\mu,\nu) =  \inf_{\pi \in \mathcal{C}(\mu,\nu)} \left(\int_{Z \times Z} d_Z(x,y)^p \pi(dx \times dy)\right)^{\frac{1}{p}},
    \]
    and the integral expression is replaced by an essential supremum in the $p = \infty$ case.
\end{definition}
The classical notion of Wasserstein distance between probability measures on the \emph{same} metric space can be adapted to compare probability measures on \emph{distinct} metric spaces---this was the setting of the initial work on \emph{Gromov-Wasserstein} distances~\citep{memoli2007,memoli2011}. The concept extends further to define metrics on the space of more general objects, called \emph{measure networks}, as was first formalized in~\citep{chowdhury2019gromov}. We recall the relevant definitions below.

\begin{definition}[Measure Network]\label{def:measure_network}
    A \define{measure network} is a triple $(X,\omega_X,\mu_X)$ such that $X$ is a Polish space, $\mu_X$ is a Borel probability measure and
    \[
        \omega_X:X \times X \to \R
    \]
    is a measurable function which we refer to as the \define{network kernel}. We frequently abuse notation and write $X$ in place of $(X,\omega_X,\mu_X)$, when the existence of a preferred network kernel and measure are clear from context.
\end{definition}

\begin{examples}
    Two of the most common sources of examples of measure networks are:
    \begin{enumerate}
        \item {\bf Metric measure spaces.} A \define{metric measure space} is a measure network such that the network kernel is a metric (which induces the given topology on $X$).
        \item {\bf Graph representations.} Given a (possibly weighted) graph, one can construct a measure network representation by taking $X$ to be the vertex set, $\mu_X$ some probability measure on the vertices (e.g., uniform or degree-weighted), and $\omega_X$ a graph kernel which encodes connectivity information. Natural choices of graph kernel include a (weighted) adjacency function~\citep{xu2019scalable} or a graph heat kernel~\citep{chowdhury2021generalized}.
    \end{enumerate}
\end{examples}

Measure networks can be compared via the pseudometric defined by~\citet{memoli2007,chowdhury2019gromov}, which we now recall.

\begin{definition}[Gromov-Wasserstein $p$-Distance]\label{def:gromov_wasserstein}
    Consider measure networks $ (X,\omega_X,\mu_X)$ and $ (Y,\omega_Y,\mu_Y)$, and let $p \in [1,\infty]$. The \define{$p$-distortion} of a coupling $\pi \in \mathcal{C}(\mu_X,\mu_Y)$ is
    \[
    \mathrm{dis}_p(\pi) = \|\omega_X - \omega_Y\|_{L^p(\pi \otimes \pi)},
    \]
    where $\pi \otimes \pi$ is the product measure on $(X \times Y)^2$, and where we consider $(x,y,x',y') \mapsto \omega_X(x,x') - \omega_Y(y,y')$ as a function on this space. Explicitly, for $p < \infty$ this is given by
    \[
        \mathrm{dis}_p(\pi) = \left(\int_{X \times Y} \int_{X \times Y} |\omega_X(x,x') - \omega_Y(y,y')|^p \pi(dx \times dy) \pi(dx' \times dy')\right)^{1/p},
    \]
    and for $p = \infty$ it is given by the essential supremum
    \[
    \mathrm{dis}_\infty(\pi) = \mathrm{esssup}_{\pi \otimes \pi} \; |\omega_X - \omega_Y| =  \sup_{(x,y),(x',y') \in \mathrm{supp}(\pi)} \; |\omega_X(x,x') - \omega_Y(y,y')|,
    \]
    where $\mathrm{supp}(\pi)$ denotes the support of the measure. Note that we suppress the dependence on $X$ and $Y$ from our notation for the distortion function. 
    The \define{Gromov-Wasserstein (GW) $p$-distance} between $X$ and $Y$ is
    \[
        \dgw (X,Y) = \frac{1}{2} \inf_{\pi \in \mathcal{C}(\mu_X,\mu_Y)} \disp (\pi).
    \]
\end{definition}

In this article we will introduce a generalization of the GW distance and when clarity is needed, we refer to $\dgw$ as the \define{\emph{standard} Gromov-Wasserstein (GW) $p$-distance}. The standard GW $p$-distance is finite when restricted to the space of measure networks whose network kernels have finite $p$-th moment. On this subspace, it defines a pseudometric whose distance-zero equivalence classes can be precisely characterized: see \citep[Theorem 2.4]{chowdhury2019gromov} or the discussion in \Cref{sec: metric_properties} below.

\subsubsection{Metric Space-Valued \texorpdfstring{$L^p$}{Lp}-Spaces}\label{sec:metric_space_valued_Lp}

The main object of study in this paper generalizes the above definitions to the case where the network kernel $\omega$ takes values in a general metric space. To ensure that the distance is finite, we will assume that the network kernel is an $L^p$ function, in the sense of \citet[Section 1.1]{Korevaar1993SobolevSA}. This parallels the definition of the Wasserstein distance, where we assume that the measures have finite $p$-th moments. 

We will now recall the definition of $L^p$ functions valued in a metric space. For measurable functions $f$ and $g$, the natural definition of a distance between $f$ and $g$ is the integral of $d(f(x),g(x))$, but this function is not always measurable due to Nedoma's pathology \citep[Section 15.9]{SchillingKuhn2021}. However, if we assume that the target metric space is separable, the distance is measurable, and the integral is well-defined. With this in mind, we present the following definition, from~\citet[Section 1.1]{Korevaar1993SobolevSA}.

\begin{definition}[{Metric Space-Valued $L^p$ Spaces}]
\label{def:metric_lp_spaces}
Let $X$ be a measure space equipped with a measure $\mu_X$, and suppose $(Y,d_Y)$ is a separable metric space. Fix $y_0\in Y$. We define the \define{space of $L^p$-functions $X \to Y$} by
    \begin{equation} 
        L^p(X,\mu_X; Y) = \left\{f:X\to Y \middle| \int_{X}d_Y(f(x),y_0)^p\mu_X(dx)<\infty \right\}
    \end{equation}
for $p \in [1,\infty)$; similarly, for $p = \infty$, the space is defined by
\begin{equation} 
    L^\infty(X,\mu_X; Y) = \left\{f:X \to Y \middle| 
\inf \{C \geq 0 \mid d_Y(f(x),y_0) \leq C \mbox{ for $\mu_X$-a.e. $x \in X$} \} < \infty \right\}.
\end{equation}
These can be expressed somewhat more concisely, for all $p \in [1,\infty]$, as 
\[
L^p(X,\mu_X;Y) = \{f:X \to Y \mid \|d_Y(f(\cdot),y_0)\|_{L^p(\mu_X)} < \infty\}.
\]
The $L^p$ space is then equipped with the distance
    \[
        D_p(f,g)         = \|d_Y(f(x),g(x))\|_{L^p(\mu_X)} = \left\{\begin{array}{cl}
        \left(\int_{X}d_Y(f(x),g(x))^p\mu_X(dx)\right)^{1/p} & 1\leq p<\infty \\
     \textnormal{esssup}_{\mu_X}d_Y(f(x),g(x)) & p=\infty. \end{array}\right.
   \]
\end{definition}

\begin{remark}
    If $f \in L^p(X,\mu_X;Y)$, then $\|d_Y(f(\cdot),y)\|_{L^p(\mu_X)} < \infty$ holds for every $y \in Y$. This is an easy consequence of the triangle inequalities of $d_Y$ and the $L^p(\mu_X)$-norm. Therefore, the above definition of $L^p$ spaces is independent of the choice of $y_0$.
\end{remark}

The distance $D_p$ is a pseudometric which assigns distance zero to functions that agree almost everywhere. We abuse notation and consider $L^p(X,\mu_X; Y)$ as a metric space by implicitly identifying functions that differ only on a measure zero set. The following properties of $L^p$ spaces will be useful later to establish various properties of the metrics introduced in this paper---see \Cref{prop: separability}, \Cref{prop:dense_subset} and \Cref{thm:complete}. The proof is provided in \Cref{app:lp complete_separable}. 

\begin{proposition}
    \label{prop: lp complete_separable}
    Let $X$ be a measure space equipped with a measure $\mu_X$, and suppose $(Y,d_Y)$ is a separable metric space. Then:
    \begin{enumerate}
        \item $L^p(X,\mu_X; Y)$ is a complete metric space if $Y$ is complete.
        \item $L^p(X,\mu_X; Y)$ is separable. \label{prop: lp_separable}
        \item Any $f\in L^p([0,1]^d,\mathscr{L}^d;Y)$ can be approximated by piecewise constant functions that are constant on a grid of a fixed step size. Here, $\mathscr{L}^d$ is the Lebesgue measure on~$[0,1]^d$. \label{prop: lp_grid_approx}
    \end{enumerate}
\end{proposition}

\begin{remark}
    \Cref{prop: lp_separable} and \Cref{prop: lp_grid_approx} in \Cref{prop: lp complete_separable} are classical results when $Y$ is a Banach space. However, to the best of our knowledge, the case when $Y$ is a general separable metric space has not been studied in the prior literature.
\end{remark}

\subsection{\texorpdfstring{$Z$}{Z}-Valued Measure Networks}

In this subsection, we introduce our main objects of study, which generalize the notion of a measure network (\Cref{def:measure_network}) and an extension of GW distance for comparing them. The basic idea is to allow network kernels to take values in some fixed, but arbitrary metric space. Similar ideas were considered recently in the machine learning literature~\citep{yang2023exploiting,kawano2024multi}, with applications to analysis of attributed graph data sets, but the idea goes back at least to~\citet{jain2009structure}---see the discussion in Section~\ref{sec:related_work}. The definitions given below generalize those which have appeared in the previous literature.

\subsubsection{Main Definitions}
Throughout this section, and much of the rest of the paper, we fix a complete and separable metric space $(Z,d_Z)$, which we frequently refer to only as $Z$. 

\begin{definition}[$Z$-Network]
    A \define{$Z$-valued $p$-measure network} is a triple $(X,\omega_X,\mu_X)$, where $X$ is a Polish space, $\mu_X$ is a Borel probability measure and
    \[
        \omega_X:X \times X \to Z
    \]
    is an element of $L^p(X\times X, \mu_X\otimes \mu_X; Z)$. We refer to $\omega_X$ as a \define{network kernel}, or \define{$Z$-valued $p$-network kernel}, when additional clarity is necessary. We frequently abuse notation and write $X$ in place of $(X,\omega_X,\mu_X)$. For short, we refer to $X$ as a \define{$(Z,p)$-network}; if the particular value of $p$ is not important to the discussion, we omit $p$ and refer to $X$ as a \define{$Z$-network}.
\end{definition}

We now introduce a notion of distance between $Z$-networks.

\begin{definition}[Gromov-Wasserstein Distance for $(Z,p)$-Networks]\label{def:ZGW_distance}
     Let $p \in [1,\infty]$ and let $X = (X,\omega_X,\mu_X)$ and $Y = (Y,\omega_Y,\mu_Y)$ be $Z$-valued $p$-networks. The associated \define{$p$-distortion} of a coupling $\pi \in \mathcal{C}(\mu_X,\mu_Y)$ is
    \[
        \mathrm{dis}^Z_p(\pi) = \|d_Z \circ (\omega_X \times \omega_Y) \|_{L^p(\pi \otimes \pi)}.
    \]
    Explicitly, for $p < \infty$,
    \[
        \mathrm{dis}^Z_p(\pi) = \left(\int_{X \times Y} \int_{X \times Y} d_Z(\omega_X(x,x'),\omega_Y(y,y'))^p \pi(dx \times dy) \pi(dx' \times dy')\right)^{1/p}
    \]
    and, for $p = \infty$,
    \[
        \mathrm{dis}^Z_\infty(\pi) = \mathrm{esssup}_{\pi \otimes \pi} d_Z \circ (\omega_X \times \omega_Y).
    \]
    The \define{$Z$-Gromov-Wasserstein ($Z$-GW) $p$-distance} between $X$ and $Y$ is
    \[
        \dgwz (X,Y) = \frac{1}{2} \inf_{\pi \in \mathcal{C}(\mu_X,\mu_Y)} \disp^Z (\pi).
    \]
    When the metric on $Z$ needs to be emphasized, we write this as $\mathrm{GW}_p^{(Z,d_Z)}$. 
\end{definition}

Important examples of $Z$-networks are provided below in \Cref{sec:examples}. Sections \ref{sec: metric_properties} and \ref{sec:approximations} will be dedicated to rigorously developing and unveiling the fundamental properties of the $Z$-GW distance.

\section{Examples of \texorpdfstring{$Z$}{Z}-Networks and \texorpdfstring{$Z$}{Z}-Gromov-Wasserstein Distances}\label{sec:examples}

We show in \Cref{sec: metric_properties} below that the $Z$-Gromov-Wasserstein distance from \Cref{def:ZGW_distance} induces a metric on a certain quotient of the space of $Z$-networks and establish some of its geometric properties. To motivate this, we first provide many examples of $Z$-GW distances; these include metrics which have already appeared in the literature (\Cref{sec:existing_ZGW}) and other natural metrics that appear to be novel (\Cref{sec:other_examples}). An important takeaway message is that it is not always necessary to re-establish metric properties of variants of GW distances from scratch---a practice which is common in the recent literature---since these properties frequently follow immediately from high-level principles. 

\subsection{Examples of Existing \texorpdfstring{$Z$}{Z}-Gromov-Wasserstein Distances in the Literature}\label{sec:existing_ZGW}
In this subsection, we show that the $Z$-GW distance, as defined above, generalizes several optimal transport distances which have previously appeared in the literature. Our main results are summarized in the following theorem (see also \Cref{tab:metrics}):
\begin{theorem}\label{thm:ZGW_generalizes}
    For appropriate choices of $Z$, the following distances can be realized as $Z$-Gromov-Wasserstein distances: Wasserstein distance, standard GW distance~\citep{memoli2007}, ultrametric GW distance~\citep{memoli2021ultrametric}, $(p,q)$-GW distance~\citep{arya2023gromov}, Fused GW distance~\citep{vayer2020fused}, Fused Network GW distance~\citep{yang2023exploiting}, spectral GW distance~\citep{memoli2009spectral}, and GW distance between weighted dynamic metric spaces~\citep{woojin-thesis}.
\end{theorem}
The definitions of these optimal transport-type distances are provided below, as necessary, and the proof of the theorem is split  among Propositions \ref{prop:GWasZGW}, \ref{prop:WasZGW}, \ref{prop:uGWasZGW}, \ref{prop:pqGWasZGW}, \ref{prop:FNGW_equals_ZGW}, \ref{prop:specGWasZGW}, and \ref{prop:DMS}, which treat the various metrics independently.

\subsubsection{Wasserstein Distance and Gromov-Wasserstein Distances}
The first result is obvious from the definitions. 
\begin{proposition}\label{prop:GWasZGW}
    The standard Gromov-Wasserstein distance between measure networks (\Cref{def:gromov_wasserstein}) is a $Z$-GW distance with $(Z,d_Z)$ equal to $\R$ with its standard metric. 
\end{proposition}

Consider measures $\mu$ and $\nu$ on a metric space $(Z,d_Z)$. It is well-known that the standard GW distance is not a generalization of the Wasserstein distance, in the sense that, for $X = (Z,d_Z,\mu)$ and $Y = (Z,d_Z,\nu)$, 
    \begin{equation}
     \mathrm{GW}_p(X,Y) \leq \mathrm{W}_p(\mu,\nu),
    \end{equation}
(see \citealt[Theorem 5.1(c)]{memoli2011}), but equality does not hold in general---as a simple example, consider the case where $Z = \R$ and $\mu$ and $\nu$ are translates of one another, yielding $\mathrm{GW}_p(X,Y) = 0$ and $\mathrm{W}_p(\mu,\nu) > 0$. Next, we show that $Z$-GW distance does give such a generalization.

\begin{proposition}\label{prop:WasZGW}
    The Wasserstein distance over an arbitrary metric space $(Z,d_Z)$ (\Cref{def:wasserstein}) can be realized as a $Z$-GW distance.
\end{proposition}

\begin{proof}
    Let $(Z,d_Z)$ be a metric space and let $\mu$ and $\nu$ be Borel probability measures on $Z$ with finite $p$-th moments. To account for different scaling conventions in the definitions of Wasserstein and GW distances, we replace $d_Z$ with $\hat{d}_Z = 2 \cdot d_Z$ as the metric on the target for our $Z$-networks. Define $Z$-valued $p$-measure networks $X = (Z,\omega_X,\mu)$ and $Y = (Z,\omega_Y,\nu)$ with $\omega_X, \omega_Y:Z \times Z \to Z$ both denoting projection onto the first coordinate, i.e., $\omega_X(z,z') = \omega_Y(z,z') = z$. Then, for any coupling $\pi \in \mathcal{C}(\mu,\nu)$, we have, for $p < \infty$,
    \begin{align} 
        \frac{1}{2^p} \mathrm{dis}_p^Z(\pi)^p & = \frac{1}{2^p}\int_{X \times Y} \int_{X \times Y} \hat{d}_Z(\omega_X(x,x'),\omega_Y(y,y'))^p \pi(dx \times dy) \pi(dx' \times dy') \\
        & = \frac{1}{2^p}\int_{X \times Y} \int_{X \times Y} 2^p \cdot d_Z(x,y)^p \pi(dx \times dy) \pi(dx' \times dy') \\
                     & = \int_{X \times Y} d_Z(x,y)^p \pi(dx \times dy).
    \end{align}
    The cost of any coupling is therefore the same for the $Z$-GW and Wasserstein distances, and it follows that $\dgwz(X,Y) = \mathrm{W}_p(\mu,\nu)$. We wrote out the calculation using integrals for the sake of clarity, but re-writing in terms of norms shows that it extends to the $p=\infty$ case without change.
\end{proof}

The \emph{ultrametric Gromov-Wasserstein} distances were introduced by~\citet{memoli2021ultrametric} as a way to compare \define{ultrametric measure spaces}---that is, measure networks $(X,\omega_X,\mu_X)$ such that the network kernel is a metric on $X$ which additionally satisfies the \emph{strong triangle inequality} $\omega_X(x,x'') \leq \max\{\omega_X(x,x'),\omega_X(x',x'')\}$. For $p \in [1,\infty)$ (or with the obvious extension if $p = \infty$), the \define{ultrametric Gromov-Wasserstein} distance between ultrametric measure spaces $X$ and $Y$ is 
\begin{equation}
\mathrm{GW}_{p,\infty}(X,Y) = \frac{1}{2} \left(\iint_{(X \times Y)^2}  \Lambda_\infty(\omega_X(x,x'),\omega_Y(y,y'))^p \pi(dx \times dy) \pi(dx' \times dy') \right)^{1/p},
\end{equation}
where $\Lambda_\infty:\R_{\geq 0} \times \R_{\geq 0} \to \R$ is the function, 
\begin{equation}\label{eqn:lambda_infty}
\Lambda_\infty(a,b) = \left\{\begin{array}{cl}
\max\{a,b\} & a \neq b \\
0 & a = b. \end{array}\right.
\end{equation}
It is easy to check that $\Lambda_\infty$ is a metric, so that $\mathrm{GW}_{p,\infty} = \mathrm{GW}_p^{(\R_{\geq 0},\Lambda_\infty)}$. We have proved the following:

\begin{proposition}\label{prop:uGWasZGW}
    The ultrametric Gromov-Wasserstein distance introduced by~\citet{memoli2021ultrametric} is a $Z$-GW distance with $(Z,d_Z) = (\R_{\geq 0}, \Lambda_\infty)$.
\end{proposition}

Analogously to \eqref{eqn:lambda_infty}, one can define a family of functions  $\Lambda_q:\R_{\geq 0} \times \R_{\geq 0} \to \R$, for $q \in [1,\infty)$, by 
\begin{equation}\label{eqn:q_metrics}
\Lambda_q(a,b) = |a^q - b^q|^{1/q}. 
\end{equation}
This leads to the \define{$(p,q)$-Gromov-Wasserstein distance} of~\citet{arya2023gromov}:
\begin{equation}\label{eqn:GWpq}
\mathrm{GW}_{p,q}(X,Y) = \frac{1}{2}\inf_{\pi \in \mathcal{C}(\mu_X,\mu_Y)} \left(\iint_{(X \times Y)^2} \Lambda_q(\omega_X(x,x'),\omega_Y(y,y'))^p \pi(dx \times dy) \pi(dx' \times dy') \right)^{1/p}
\end{equation}
(defined here for $p < \infty$, with the $p=\infty$ version defined similarly). The functions $\Lambda_q$ are metrics on $\R_{\geq 0}$ \citep[Proposition 1.13]{arya2023gromov} and we clearly have $\mathrm{GW}_{p,q} = \mathrm{GW}_p^{(\R_{\geq 0},\Lambda_q)}$. Thus we have shown:

\begin{proposition}\label{prop:pqGWasZGW}
    The $(p,q)$-Gromov-Wasserstein distance of~\citet{arya2023gromov} is a $Z$-GW distance with $(Z,d_Z) = (\R_{\geq 0}, \Lambda_q)$. 
\end{proposition}

\begin{remark}\label{rem:comparison_to_sturm}
    ~\citet{sturm2023space} considers a distance very similar to the $(p,q)$-GW distance defined in \eqref{eqn:GWpq}. The distinction is that the integrand in Sturm's version is raised to a $q$th power (i.e., it integrates the function $\Lambda_q(\omega_X(\cdot,\cdot),\omega_Y(\cdot,\cdot))^{qp}$). As was observed by~\citet{arya2023gromov}, this version does not enjoy the same homogeneity under scaling that \eqref{eqn:GWpq} does. Moreover, we prefer the formulation of \eqref{eqn:GWpq}, since $\Lambda_q$ defines a metric, but $\Lambda_q^q$ does not. 
\end{remark}

\subsubsection{Attributed Graphs}\label{sec:attributed_graphs}

In practice, graph data sets are commonly endowed with additional attributes. As a concrete example (from~\citealt{kawano2024multi}), consider a graphical representation of a molecule, where nodes represent atoms and edges represent bonds. Each node and edge is then naturally endowed with a categorical label---the atom type and bond type, respectively. This information can be one-hot encoded, leading to an $\R^n$-valued label on each node and an $\R^m$-valued label on each edge ($n$ and $m$ being the number of atom and bond types present in the data set, respectively). Attributed graphs are also ubiquitous in modeling social networks (see, e.g.,~\citealp{bothorel2015clustering}), where node attributes encode statistics about members of the network and edge attributes can record diverse information about member interactions---in this setting, data typically includes (non-binary) attributes in Euclidean spaces. Graphs with attributes in more exotic metric spaces are also of interest---we describe examples from~\citet{bal2022statistical} and~\citet{robertson2023generalization} in detail below, in \Cref{sec:other_examples}. 

A general model for graphs with metric space-attributed nodes and edges is introduced in the work of \citet{yang2023exploiting}. Following their ideas, we consider the following structure.

\begin{definition}[Attributed Network]\label{def:attributed_network}
    We work with the following collection of hyperparameters $\mathcal{H} = (p,\Omega,d_\Omega,\Psi,d_\Psi)$, where $p \in [1,\infty]$, and $(\Omega,d_\Omega)$, $(\Psi,d_\Psi)$ are separable metric spaces. An \define{attributed network with hyperparameter data $\mathcal{H}$}, or simply \define{$\mathcal{H}$-network}, is a 5-tuple of the form $(X,\psi_X,\phi_X,\omega_X,\mu_X)$, where
    \begin{itemize}
        \item the triple $(X,\phi_X,\mu_X)$ is a measure network (valued in $\R$), with $\phi_X \in L^p(\mu_X\otimes \mu_X)$, which models the underlying graph of an attributed graph,
        \item $\psi_X \in L^p(X,\mu_X; \Psi)$ models node features attributed in $\Psi$, and 
        \item $\omega_X \in L^p(X \times X, \mu_X \otimes \mu_X; \Omega)$ models edge features attributed in $\Omega$.
    \end{itemize}
\end{definition}

\begin{remark}
    The definition above is more general than the one given by~\citet[Definition 2.1]{yang2023exploiting}: therein, all functions are assumed to be bounded and continuous. Moreover, \citet{yang2023exploiting} does not assume separability of the target metric spaces, which can potentially lead to technical issues of well-definedness, following the discussion in \Cref{sec:metric_space_valued_Lp}.
\end{remark}

A notion of distance between $\mathcal{H}$-networks is also given by~\citet{yang2023exploiting}, as we now recall.

\begin{definition}[Fused Network Gromov-Wasserstein Distance, \cite{yang2023exploiting}]\label{def:fused_network_GW}
    For fixed hyperparameter data $\mathcal{H} = (p,\Omega,d_\Omega,\Psi,d_\Psi)$, let  
    \[
    X = (X,\psi_X,\phi_X,\omega_X,\mu_X)\quad  \mbox{and}\quad  Y = (Y,\psi_Y,\phi_Y,\omega_Y,\mu_Y)
    \]
    be $\mathcal{H}$-networks. Given additional hyperparameters $q \in [1,\infty)$ and $\alpha,\beta \in [0,1]$ with $\alpha + \beta \leq 1$, the associated \define{fused network Gromov-Wasserstein (FNGW) distance} is 
    \begin{multline}\label{eqn:fused_network_GW}
        \mathrm{FNGW}^\mathcal{H}_{q,\alpha,\beta}(X,Y) = \frac{1}{2} \inf_{\pi \in \mathcal{C}(\mu_X,\mu_Y)} \left(\int_{X \times Y} \int_{X \times Y} \left((1-\alpha-\beta)\,d_\Psi(\psi_X(x),\psi_Y(y))^q \right.\right. \\
        \left.\left. + \alpha \, d_\Omega(\omega_X(x,x'),\omega_Y(y,y'))^q + \beta \,|\phi_X(x,x') - \phi_Y(y,y')|^q\right)^{p/q} \pi(dx \times dy) \pi(dx' \times dy')\right)^{1/p}
    \end{multline}
    The definition extends to the case $p=\infty$ or $q=\infty$ cases straightforwardly.
\end{definition}

\begin{remark}\label{rem:qth_root}
    The formulation given in \eqref{eqn:fused_network_GW} is actually subtly different than the one given by~\citet{yang2023exploiting}. In analogy with \Cref{rem:comparison_to_sturm}, we have added an extra power of $1/q$ to the integrand. We will see below that this leads to improved theoretical properties of the distance.
\end{remark}

Observe that the distance \eqref{eqn:fused_network_GW} specializes to other notions of distance in the literature through appropriate choices of the balance parameters $\alpha$ and $\beta$:
\begin{enumerate}
    \item Taking $\alpha = \beta = 0$, one obtains the Wasserstein $p$-distance on $(\Psi,d_\Psi)$ between the pushforward measures $(\psi_X)_\ast \mu_X$ and $(\psi_Y)_\ast \mu_Y$. 
    \item Taking $\alpha = 1$, $\beta = 0$, we recover the $Z$-GW $p$-distance (\Cref{def:ZGW_distance}) between the underlying $Z$-networks, with $(Z,d_Z) = (\Omega,d_\Omega)$. Conversely, we show below that FNGW can be formulated as a special case of a $Z$-GW distance (\Cref{prop:FNGW_equals_ZGW}).
    \item Taking $\alpha = 0$, $\beta = 1$, recovers the standard GW distance between the underlying measure networks.
    \item Finally, taking general $\alpha = 0$ and $\beta \in [0,1]$, we get the \define{Fused Gromov-Wasserstein (FGW) distance}, which was introduced by~\citet{vayer2020fused} as a framework for comparing graphs with only node attributes. 
\end{enumerate}

\begin{remark}\label{rem:qth_root_FGW}
    The original formulation of FGW distance also omitted the $q$-th root in the integrand that was discussed in \Cref{rem:qth_root}, but we argue for its inclusion for theoretical reasons discussed below. The FGW framework introduced by~\citet{vayer2020fused} is also slightly more general, in that it considers a joint distribution on $X \times \Psi$, rather than a function $\psi_X:X \to \Psi$; the formulation considered here is the one most commonly used in practice---see~\citep{titouan2019optimal,flamary2021pot}.
\end{remark}

We now show that the rather complicated formulation of \eqref{eqn:fused_network_GW} is somewhat redundant.

\begin{proposition}\label{prop:FNGW_equals_ZGW}
An attributed network with hyperparameter data $\mathcal{H}$ has a natural representation as a $(Z,p)$-network, for an appropriate choice of $(Z,d_Z)$. Moreover, the Fused Network Gromov-Wasserstein distance  \eqref{eqn:fused_network_GW} can be realized as a $Z$-Gromov-Wasserstein distance.
\end{proposition}

\begin{proof}
    Let hyperparameters $\mathcal{H} = (p,\Omega,d_\Omega,\Psi,d_\Psi)$, as well as $q \in [1,\infty)$ and $\alpha,\beta \in [0,1]$ with $\alpha + \beta \leq 1$ be given. Suppose that  $X = (X,\psi_X,\phi_X,\omega_X,\mu_X)$ and $Y = (Y,\psi_Y,\phi_Y,\omega_Y,\mu_Y)$ are $\mathcal{H}$-networks. We construct a metric space $Z$ as follows. Let $Z = \Psi \times \Omega \times \R$ and let $d_Z$ be the weighted $\ell^q$ metric
    \[
    d_Z((a,b,c),(a',b',c')) = \left((1-\alpha-\beta)d_\Psi(a,a')^q + \alpha d_\Omega(b,b')^q + \beta |c - c'|^q\right)^{1/q}.
    \]

    Now, we define a $Z$-network $\overline{X} = (X,\overline{\omega}_X,\mu_X)$ by setting
    \[
    \overline{\omega}_X(x,x') = \left(\psi_X(x),\omega_X(x,x'),\phi_X(x,x')\right) \in \Psi \times \Omega \times \R = Z.
    \]
    We define $\overline{Y} = (Y,\overline{\omega}_Y,\mu_Y)$ similarly. Then, for any $\pi \in \mathcal{C}(\mu_X,\mu_Y)$, we have 
    \begin{align} 
         &\iint d_Z(\overline{\omega}_X(x,x'), \overline{\omega}_Y(y,y'))^p \pi(dx \times dy) \pi(dx' \times dy') \\
        &\qquad \qquad = \iint \left((1-\alpha-\beta)d_\Psi(\psi_X(x),\psi_Y(y))^q + \alpha d_\Omega(\omega_X(x,x'),\omega_Y(y,y'))^q \right.\\
        &\hspace{1.5in} \left. + \beta |\phi_X(x,x') - \phi_Y(y,y')|^q \right)^{p/q} \pi(dx \times dy)\pi(dx' \times dy'),
    \end{align}
    and it follows that $\mathrm{GW}_p^Z(\overline{X},\overline{Y}) = \mathrm{FNGW}^\mathcal{H}_{q,\alpha,\beta}(X,Y)$.
\end{proof}

\subsubsection{Diffusions and Stochastic Processes}

Inspired by a notion of spectral convergence for Riemannian manifolds introduced by \citet{kasue1994spectral}, \citet{memoli2009spectral,memoli2011spectral} considers a certain \emph{spectral} version of the GW distance between Riemannian manifolds (see S. Lim's PhD thesis, \citealt[Chapter 5]{thesis-sunhyuk}, for a generalization to Markov processes and \citealt{chen2022weisfeiler,chen2023weisfeiler} for related ideas). For a given compact Riemannian manifold $(M,g_M)$ let $k_M:M\times M\times \R_{>0}\to \R_{>0}$ denote its normalized heat kernel\footnote{i.e., $\lim_{t\to\infty}k_M(x,x',t) = 1$ for all $x,x'\in M$.} and $\mu_M$ its normalized volume measure (we use $\R_{>0}$ to denote the set of positive real numbers). Then, if $N$ is another compact Riemannian manifold and $p\geq 1$, the \define{spectral Gromov-Wasserstein $p$-distance} between $M$ and $N$ is
\[
\dgw^\mathrm{spec}(M,N):=\frac{1}{2} \inf_{\pi\in\mathcal{C}(\mu_X,\mu_Y)}\sup_{t>0} c^2(t)\cdot\|\Gamma_{M,N,t}^\mathrm{spec}\|_{L^p(\pi\otimes \pi)}
\]
where $c(t):=e^{-t^{-1}}$ and 
$\Gamma_{M,N,t}^\mathrm{spec}(x,y,x',y'):=\big|k_M(x,x',t)-k_N(y,y',t)\big|$ for $x,x'\in M$ and $y,y'\in Y$.
The reason for the use of the dampening function $c(t)$ is to tame the blow up of the heat kernel as $t\to 0$:  $k_M(x,x',t)\sim t^{-\tfrac{m}{2}}$  where $m$ is the dimension of $M$; see  \citep{rosenberg1997laplacian}. 

To recast $\dgw^\mathrm{spec}$ (or, rather, a variant thereof---see below) as a $Z$-GW distance, we make the following choices: 
\begin{equation}\label{eqn:specGWspace}
Z = \left\{f:\R_{>0}\to\R_{>0}\mid\sup_{t>0} c^{2}(t)|f(t)-1|<\infty\right\}
\end{equation}
and, for $f_1,f_2\in Z$,
\begin{equation}\label{eqn:specGWdistance}
d_Z(f_1,f_2)=\sup_{t>0}c^{2}(t)\big|f_1(t)-f_2(t)\big|.
\end{equation}
We then represent a compact Riemannian manifold $(M,g_M)$ as a $Z$-measure network $(M,\omega_M,\mu_M)$ with $\mu_M$ as above and $\omega_M:M \times M \to Z$ defined by
\[
\omega_M(x,x') = k_M(x,x',\cdot) \in Z, \quad \mbox{where} \quad  k_M(x,x',\cdot):t \mapsto k_M(x,x',t).
\]
Then
\begin{align} 
d_Z \circ (\omega_M \times \omega_N) (x,y,x',y') &= \sup_{t>0}c^{2}(t)\big|k_M(x,x',t)-k_N(y,y',t)\big| \\
&= \sup_{t > 0} c^2(t) \Gamma^{\mathrm{spec}}_{M,N,t}(x,y,x',y'),
\end{align}
so that, for an arbitrary coupling $\pi \in \mathcal{C}(\mu_X,\mu_Y)$,
\begin{align}
     \|d_Z \circ (\omega_M \times \omega_N)\|_{L^p(\pi \times \pi)} 
    &=  \left\| \sup_{t > 0} c^2(t) \Gamma^{\mathrm{spec}}_{M,N,t}\right\|_{L^p(\pi \otimes \pi)}.
\end{align}
Accordingly, we define a variant of the spectral GW $p$-distance:
\[
\widetilde{\mathrm{GW}}_p^\mathrm{spec}(M,N) := \frac{1}{2} \inf_{\pi \in \mathcal{C}(\mu_X,\mu_Y)} \left\| \sup_{t > 0} c^2(t) \Gamma^{\mathrm{spec}}_{M,N,t}\right\|_{L^p(\pi \otimes \pi)}.
\]
We have the following result.

\begin{proposition}\label{prop:specGWasZGW}
    The spectral GW $p$-distance $\widetilde{\mathrm{GW}}_p^\mathrm{spec}$ is a $Z$-GW distance, with $(Z,d_Z)$ as in \eqref{eqn:specGWspace} and \eqref{eqn:specGWdistance}, which upper bounds the spectral Gromov-Wasserstein distance introduced by \citet{memoli2009spectral}.
\end{proposition}

\begin{proof}
    The fact that $\widetilde{\mathrm{GW}}_p^\mathrm{spec}$ is a $Z$-GW distance follows immediately by definition. To verify the upper bound claim, let $M$ and $N$ be compact Riemannian manifolds and choose an arbitrary coupling $\pi$ between their normalized volume measures. We then consider the map
    \begin{align}
    F: (X \times Y \times X \times Y) \times \R_{> 0} &\to \R \\
    \big((x,y,x',y'),t\big) &\mapsto c^2(t) \Gamma^\mathrm{spec}_{M,N,t}(x,y,x',y')
    \end{align}
    as a function on the product of measure spaces $(X \times Y \times X \times Y, \pi \otimes \pi)$ and $(\R,\mathcal{L})$, where we use $\mathcal{L}$ to denote Lebesgue measure on $\R_{>0}$ through the duration of the proof. We now compare the objective functions of the optimization problems associated to the distances $\widetilde{\mathrm{GW}}_p^\mathrm{spec}(M,N)$ and $\mathrm{GW}_p^\mathrm{spec}(M,N)$ as
    \begin{align}
     \left\|\sup_{t > 0} c^2 (t) \Gamma^{\mathrm{spec}}_{M,N,t} \right\|_{L^p(\pi \otimes \pi)} &= \left\| \left\|F\right\|_{L^\infty(\mathcal{L})} \right\|_{L^p(\pi \otimes \pi)} \\
     &\geq \left\| \left\|F\right\|_{L^p(\pi \otimes \pi)} \right\|_{L^\infty(\mathcal{L})} \\
     &= \sup_{t > 0} \|c^2(t) \Gamma_{M,N,t}^{\mathrm{spec}}\|_{L^p(\pi \otimes \pi)} \\
     &= \sup_{t > 0} c^2(t) \cdot \|\Gamma_{M,N,t}^{\mathrm{spec}}\|_{L^p(\pi \otimes \pi)},
    \end{align}
    where we have interchanged suprema and $L^\infty$ norms via continuity in $t$, and we have applied a generalized version of Minkowski's inequality~\citep[Proposition 1.3]{bahouri2011fourier}. Since the estimate holds for an arbitrary coupling $\pi$, this shows that $\widetilde{\mathrm{GW}}_p^\mathrm{spec}(M,N) \geq \mathrm{GW}_p^\mathrm{spec}(M,N)$.
\end{proof}

\subsubsection{The Gromov-Wasserstein Distance Between Dynamic Metric Spaces}

A \define{dynamic metric space} (or DMS, for short) is a pair $(X,d_X)$, where $X$ is a finite set and $d_X:\R\times X\times X\to \R_{\geq 0}$ is such that: 
\begin{itemize}
\item for every $t\in \R$, $(X,d_X(t))$ is a pseudometric space, where $d_X(t):X \times X \to \R$ is the map $d_X(t)(x,x') = d_X(t,x,x')$;
\item for any $x,x'\in X$ with $x\neq x'$ the function $d_X(\cdot)(x,x'):\R\to \R_{\geq 0}$ defined by $t \mapsto d_X(t)(x,x')$ is continuous and not identically zero.
\end{itemize}
A \define{weighted dynamic metric space} (or wDMS) is a triple $(X,d_X,\mu_X)$ such that $(X,d_X)$ is a DMS and $\mu_X$ is a fully supported probability measure on $X$. 

Dynamic metric spaces provide a mathematical structure suitable for modeling natural phenomena such as flocking and swarming behaviors \citep{sumpter2010collective}. Variants of the Gromov-Hausdorff and Gromov-Wasserstein distances were proposed by  \citet{kim2020analysis,kim2021spatiotemporal,woojin-thesis} in order to metrize the collection of all DMSs and wDMSs, respectively. We now show how wDMSs and these distances fit into the framework of Z-GW distances.

Let $C(\R,\R_{\geq 0})$ denote the collection of all continuous functions $f:\R\to \R_{\geq 0}$. For $\lambda \geq 0$ one defines (see \citealp[Definition 2.7.1]{woojin-thesis}) the \define{$\lambda$-slack interleaving distance} between $f_1,f_2\in C(\R,\R_{\geq 0})$ by
\[
d_{\lambda}(f_1,f_2):=\inf\left\{\varepsilon\in[0,\infty]|\forall t\in\R, \min_{s\in[t-\varepsilon,t+\varepsilon]} f_i(s)\leq f_j(t)+\lambda\varepsilon,\,i,j=1,2\right\}.
\]
According to \citet[Definition 2.9.5]{woojin-thesis}, the $(p,\lambda)$-Gromov-Wasserstein distance between two wDMSs $(X,d_X,\mu_X)$ and $(Y,d_Y,\mu_Y)$ is defined as
\[
\frac{1}{2}\inf_{\pi\in\mathcal{C}(\mu_X,\mu_Y)}\big\|d_\lambda \circ (d_X \times d_Y)\big\|_{L^p(\pi\otimes \pi)}
\]
(the original definition did not include the $1/2$-scaling factor, but we include it here for consistency). We consider a wDMS $(X,d_X,\mu_X)$ as a $Z$-network with $Z = C(\R,\R_{\geq 0})$ by defining a network kernel $\omega_X:X \times X \to Z$ as $\omega_X(x,x') = d_X(\cdot)(x,x')$. 
This immediately yields the following.

\begin{proposition}
    \label{prop:DMS}
    The $(p,\lambda)$-Gromov-Wasserstein distance between wDMSs~\citep{woojin-thesis} is a $Z$-GW distance with $(Z,d_Z) = (C(\R,\R_{\geq 0}),d_\lambda)$. 
\end{proposition}

\subsubsection{The Cut Distance Between Graphons}

Recall that in the work of Borgs, Chayes,
Lov\'asz, S\'os and Vesztergombi  \citep{borgs2008convergent} (see also \citealt{lovasz2012large,janson2010graphons}), a \emph{kernel} on a probability space $(\Omega,\mu)$ is any symmetric integrable function $U:\Omega\times \Omega\to\R_{\geq 0}$. A \emph{graphon} on $(\Omega,\mu)$ is any kernel $W$ with codomain $[0,1]$.
The \emph{cut norm} of $W\in L^1(\Omega,\mu)$ is the number
\[
\|W\|_{\Box,\Omega,\mu}:=\sup_{S,T\subset \Omega}\left|\iint_{S\times T}W(x,x')\,\mu(dx)\,\mu(dx')\right|,
\]
where the supremum is over measurable sets $S,T$.

Now, given two probability spaces $(\Omega,\mu)$ and $(\Omega',\mu')$, and kernels $U$ on $\Omega$ and $U'$ on $\Omega'$  one defines the \emph{cut distance} between $(\Omega,U,\mu)$ and $(\Omega',U',\mu)$ as \citep[Theorem 6.9]{janson2010graphons}
$$d_\Box(\Omega,\Omega'):=\inf_{\pi \in \mathcal{C}(\mu,\mu')}\left\|U-U'\right\|_{\Box,\Omega\times\Omega',\pi}.$$

We note that the cut distance defines a compact metric space, while GW-type metrics do not in general exhibit completeness (see \citealt[Remark 5.18]{memoli2011}), and therefore do not lead to compactness. In particular, the $\delta_1$ metric considered by Janson in \citep{janson2010graphons} is precisely a $Z$-Gromov-Wasserstein distance for $Z=[0,1]$ and it is not compact, as mentioned in \citet[p15]{janson2010graphons}. In fact, the topology induced by $\delta_1$ is strictly finer than the one induced by the cut distance. The structural discrepancy lies in how kernel values are compared: the $Z$-Gromov-Wasserstein distance aggregates  kernel values $\omega_X, \omega_Y$ via the distance function $d_Z(\omega_X,\omega_Y)$ and then considers the $L^p$ norm of this quantity. On the other hand, the cut distance aggregates kernel values by taking their differences and then computing the cut norm. 


While not exactly fitting into our framework, it seems interesting to explore the possibility of formulating an even more general setting than the one in \Cref{def:ZGW_distance} which could (precisely) encompass the cut distance. For example, considering both an abstract ``aggregator" of kernel values and  an arbitrary norm for the resulting \emph{aggregated} quantity could lead to a framework subsuming the $Z$-Gromov-Wasserstein distance as well as the cut distance.



\subsection{Further Examples of \texorpdfstring{$Z$}{Z}-Valued Measure Networks}\label{sec:other_examples}
Next, we give several more examples of naturally occurring $Z$-valued measure networks. We believe that the resulting $Z$-GW distances are novel, but point to related notions in the literature when appropriate.

\subsubsection{An Alternative Approach to Edge-Attributed Graphs}\label{sec:alternative_approach}

In \Cref{def:attributed_network}, we gave a flexible model for graphs with node and edge attributions, which we referred to as \emph{attributed networks}, following~\citet{yang2023exploiting}. However, the formalism developed there may be unsatisfactory in practice, as we now explain. Consider a directed graph $(X,E)$, where $X$ is a set of nodes and $E \subset X \times X$ is a set of directed edges (this also captures the notion of an undirected graph, which we consider as a directed graph with a symmetric edge set). Realistic graphical data frequently comes with edge attributes (see the discussion at the beginning of Section~\ref{sec:attributed_graphs})---in practice, this is given by a function from $E$ into some attribute space $(\Omega,d_\Omega)$. Now observe that the formalism described in \Cref{def:attributed_network} (which originates from~\citealt{yang2023exploiting}) models edge attributes as functions of the form $\omega_X:X \times X \to \Omega$; that is, the function is not only defined on the edge set, but on the full product space $X \times X$. This difference can be handled in an ad hoc manner: for example, if $\Omega = \R^d$, one could assign pairs $(x,x') \in (X \times X) \setminus E$ to the zero vector. Such a choice is essentially arbitrary, and does not naturally extend to a more general metric space $\Omega$. We now describe a more principled approach to handling this issue via \emph{cone metrics}.

Let $(\Omega,d_\Omega)$ be a metric space, which we will later consider as an edge attribute space. The \define{cone space}~\citep[Section 3.6.2]{burago2022course} of $\Omega$ is 
\begin{equation}\label{eqn:quotient_space_construction}
\mathrm{Con}(\Omega):=(\Omega \times \R_{\geq 0})/(\Omega \times \{0\}).
\end{equation}
For $(u,r) \in \Omega \times \R_{\geq 0}$, its equivalence class $[u,r]$ consists of points $(v,s)$ such that $(u,r) = (v,s)$ or $r = s = 0$.  The \define{cone metric} on $\mathrm{Con}(\Omega)$ is
\[
d_{\mathrm{Con}(\Omega)}([u,r],[v,s]) := \left(r^2 + s^2 - 2rs \cos(\bar{d}_\Omega(u,v))\right)^{1/2},
\]
where $\bar{d}_\Omega(u,v) = \min \{d_\Omega(u,v),\pi \}$.

Now let $(X,E)$ be a directed graph endowed with an edge attribute function $\hat{\omega}_X:E \to \Omega$ and a network kernel $\phi_X:X \times X \to \R_{\geq 0}$ such that $\phi_X(x,x') > 0$ if and only if $(x,x') \in E$; this can either be data-driven, or derived directly from the combinatorial structure of the graph by taking $\phi_X$ to be the directed binary adjacency function induced by $E$. Let $(Z,d_Z) = (\mathrm{Con}(\Omega),d_{\mathrm{Con}(\Omega)})$. For any choice of probability measure $\mu_X$, the associated \define{$Z$-network induced by $(X,E)$} is $(X,\omega_X,\mu_X)$, with $\omega_X:X \times X \to \mathrm{Con}(\Omega)$ defined by
\[
\omega_X(x,x') := \left\{\begin{array}{rl}
[\hat{\omega}(x,x'),\phi_X(x,x')] & \mbox{if } (x,x') \in E \\
\left[p_0, 0\right] & \mbox{otherwise,}
\end{array}\right.
\]
where $p_0 \in \Omega$ is a fixed but arbitrary point. The idea is that this extends the original attribute function $\hat{\omega}_X$ (whose domain is $E$) to the full product space $X \times X$, but the cone construction results in the choice of assignment for  $(x,x') \notin E$ being inconsequential. Node attributes for the graph can be handled as in \Cref{def:attributed_network}, so that the structure defined here gives a complete refinement of the previous attributed network structure.

\Needspace{18\baselineskip}
\begin{wrapfigure}[17]{r}{0.35\textwidth}
  \centering
    \includegraphics[width=0.27\textwidth]{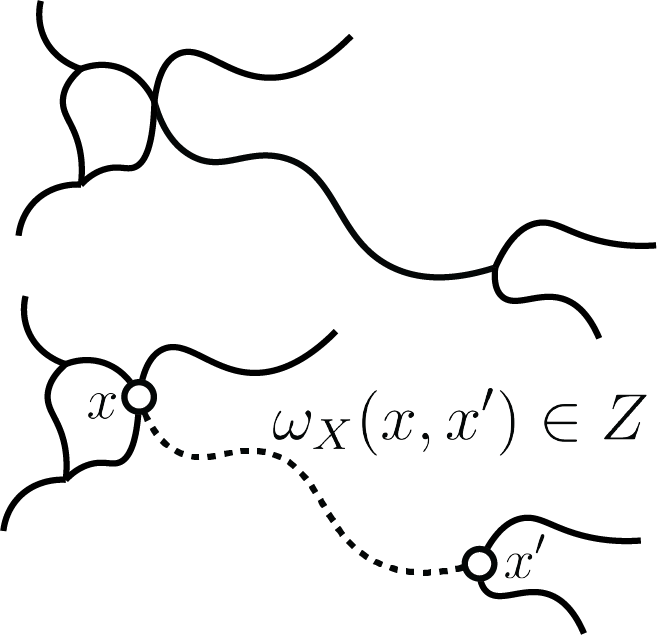}
  \vspace{.2in}
  \caption{\small{Top: Simple example of a shape graph in $\R^2$. Bottom: Representation as an attributed network. The intersection points $x$ and $x'$ (circles) are nodes in $X$, $\omega_X(x,x')$ (dashed) is an element of the space of curves $Z$.}}
  \label{fig:shape_graph}
\end{wrapfigure}

The next two subsections provide examples in the literature of $Z$-networks arising from attributed graph structures. In particular, a similar idea was used by~\citet{bal2024statistical} to model \emph{shape graphs}, described in detail below, where this setup was implemented computationally. 

 \subsubsection{Shape Graphs}\label{sec:shape_graphs}

 There is a growing body of work on statistical analysis of \emph{shape graphs}~\citep{sukurdeep2022new,liang2023shape,guo2022statistical,srivastava2020advances,bal2024statistical}---roughly, these are 1-dimensional stratified spaces embedded in some Euclidean space, relevant for modeling filamentary structures such as arterial systems or road networks (see~\Cref{fig:shape_graph}). We now explain how some models of shape graphs fit into the $Z$-GW framework. 

 Let $Z$ denote the \emph{space of unparameterized curves} in $\R^d$, where $d \in \{2,3\}$ in typical applications; that is, $Z = C^\infty([0,1],\R^d)/\mathrm{Diff}([0,1])$, the space of smooth maps of the interval into $\R^d$, denoted $C^\infty([0,1],\R^d)$, considered up to the reparameterization action of the diffeomorphism group of the interval, denoted $\textrm{Diff}([0,1])$. Let $d_Z$ be some fixed metric on $Z$; in the references considered here, this is typically induced by a diffeomorphism-invariant \emph{elastic metric}~\citep{srivastava2016functional,bauer2024elastic}, but the specific choice of metric is not important in the discussion to follow.

 \citet{guo2022statistical} modeled a shape graph as a function of the form  $\omega_X:X \times X \to Z$, where $X$ is a finite set of nodes. In practice, given a filamentary structure such as an arterial system, the nodes are the endpoints and intersection points of filaments and the value of $\omega_X(x,x') \in Z$ is the (unparameterized) filament joining the endpoints, as is illustrated in \Cref{fig:shape_graph}.  For pairs $(x,x')$ which are not joined by a curve, the authors choose to assign the constant curve taking the value $0 \in \R^d$. In an effort to overcome the shortcomings of this ad hoc assignment, the followup work~\citep{bal2024statistical} introduces a formalism similar to that of \Cref{sec:alternative_approach} for representing shape graphs, where the particular point in $Z$ assigned to missing edges is made irrelevant to distance computations via a quotient construction. Either method for modeling shape graphs results in a $Z$-network by, say, assigning the node set the uniform measure.

 In both~\citet{guo2022statistical} and \citet{bal2024statistical}, comparisons of shape graphs were done via the Jain and Obermayer framework~\citep{jain2009structure} described earlier. Specifically, for shape graphs with the same number of nodes, the distance between them is computed by aligning the $Z$-valued kernels over permutations. When comparing two shape graphs with node sets of different sizes, the smaller of the two is ``padded" with additional disconnected nodes. This padding step is arguably unnatural, and can be avoided by instead working within the $Z$-GW framework, where differently sized node sets are not an issue.

 \subsubsection{Connection Graphs}

Let $(X,E)$ be a directed graph such that $(x,x') \in E \Leftrightarrow (x',x) \in E$ (i.e., $(X,E)$ uses the directed graph formalism to encode an inherently undirected graph). A \define{connection} on $(X,E)$~\citep{robertson2023generalization} is a map $\hat{\omega}_X:E \to \mathrm{O}(d)$ such that $\hat{\omega}(x,x') = \hat{\omega}(x',x)^{-1}$, where $\mathrm{O}(d)$ is the orthogonal group of Euclidean space $\R^d$. These structures appear, for example, in alignment problems coming from cryo-electron microscopy~\citep{bhamre2015orthogonal,singer2011angular}. Let $d_{\mathrm{O}(d)}$ be the metric on $\mathrm{O}(d)$ induced by identifying it with a subspace of the matrix space $\R^{d \times d}$, endowed with Frobenius distance. For a choice of measure $\mu_X$, one obtains a $Z$-network $(X,\omega_X,\mu_X)$, with $(Z,d_Z) = (\mathrm{O}(d),d_{\mathrm{O}(d)})$ and with $\omega_X:X \times X \to Z$ defined as in~\Cref{sec:alternative_approach}.

 \subsubsection{Binary Operations} 
 
 Let $(Z,d_Z)$ be a metric space endowed with a binary operation $\bullet: Z \times Z \to Z$. For any subset $X \subset Z$, let $\omega_X:X \times X \to Z$ be defined by $\omega_X(x,x') = x \bullet x'$. A choice of probability measure $\mu_X$ results in a $Z$-network $(X,\omega_X,\mu_X)$.

 As an example of the above construction, suppose that $Z$ is a compact Lie group endowed with a bi-invariant Riemannian metric, and let $\bullet$ denote its group law. Given a compact Lie subgroup $X \subset Z$, let $\mu_X$ denote its normalized Haar measure. The $Z$-GW distance then defines a natural distance on the space of compact Lie subgroups of $Z$.

\subsubsection{Probabilistic Metric Spaces}\label{sec:probabilistic_metric_spaces}

 In his 1942 paper \citep{menger1942statistical}, Menger initiated the study of \emph{probabilistic metric spaces}; see also \citep{wald1943statistical,kramosil1975fuzzy,schweizer1960statistical}. A
\define{probabilistic metric space} is a pair $(X,p_X)$ in which the ``distance" $p_X(x,x')$
between any two points $x,x'\in X$ is a Borel probability measure  on $\R_{\geq 0}$ satisfying the following axioms:
\begin{enumerate}
\item $x=x'$ if and only if  $p_X(x,x')=\delta_0$;
\item $p_X(x,x') = p_X(x',x)$ for all $x,x'\in X$;
\item $\min\big(p_X(x,x')([0,s]),p_X(x',x
'')([0,t])\big) \leq p_X(x,x'')([0,s+t])$ for all $x,x',x''\in X$ and all $s,t\in\R_{\geq 0}$.
\end{enumerate}

Note that axiom 3 is a generalization of the triangle inequality: if $d_X$ is a metric on $X$ and distances are \emph{deterministic}, i.e., $p_X(x,x') = \delta_{d_X(x,x')}$ for all $x,x'$, then this condition is equivalent to the triangle inequality for $d_X$.

\medskip

We can recast probabilistic metric spaces as $Z$-networks for the choice 
\[
Z = \{\text{Borel probability measures on $\R_{\geq 0}$ satisfying axioms 1,2, and 3}\}
\]
and for $d_Z = \mathrm{W}_p$, the Wasserstein distance (of order $p\geq 1$) on $\R_{\geq 0}$. Probabilistic metric spaces can therefore be compared via the associated $Z$-GW distance.

\section{Properties of \texorpdfstring{$Z$}{Z}-Gromov-Wasserstein Distances}
\label{sec: metric_properties}

This section establishes basic metric properties of the $Z$-GW distances, as well as the induced topological properties of the space of $(Z,p)$-networks. As a result, we obtain a unified proof that these properties hold for the distances from the literature described in \Cref{sec:existing_ZGW} (summarized in \Cref{thm:ZGW_generalizes} and \Cref{tab:metrics}), and a simultaneous proof for the distances in the novel settings described in \Cref{sec:other_examples}.

\subsection{Basic Metric Structure}

We begin by showing that $Z$-GW distances are metrics when the space of $(Z,p)$-networks is considered up to a natural notion of equivalence.

\subsubsection{Existence of Optimal Couplings}
In this section, we will study the important question of whether the optimization problem defining the $Z$-GW distance always has a solution. Our main result of this part is the following theorem, which gives an affirmative answer in the general setting of the present article:
\begin{theorem}
    \label{thm: optimal coupling}
    For any $Z$-networks $X,Y $ and any $p \in [1,\infty]$, there exists $\pi \in \mathcal{C}(\mu_X,\mu_Y)$ such that $\dgwz(X,Y) = \frac{1}{2}\mathrm{dis}_p^Z(\pi)$. That is, optimal couplings always exist.
\end{theorem}
Our proof strategy is as follows: since $\mathcal{C}(\mu_X,\mu_Y)$ is sequentially compact with respect to the weak convergence of probability measures \citep[p. 32]{villani2003topics}, it is enough to prove that the distortion functional $\disp^Z$ is continuous in this topology. This approach was also considered in previous works such as \citet{chowdhury2019gromov} and \citet{memoli2011}, but the same results do not apply in our general setting. For example, \citet{memoli2011} heavily utilizes the fact that the kernels are metrics, and \citet{chowdhury2019gromov} use the fact that the bounded continuous $\mathbb{R}$-valued functions are dense in $L^p$ spaces, which is not true for a general target metric space (consider $Z=\{0,1\}$).  The corresponding result for the fused network GW distances (\Cref{def:fused_network_GW}) was established by~\citet{yang2023exploiting}. In light of \Cref{prop:FNGW_equals_ZGW}, and the fact that~\citet{yang2023exploiting} considered more restrictive structures with bounded and continuous kernels, our result further generalizes the result of~\citet[Theorem 2.4]{yang2023exploiting}.

To deal with our case, we utilize the following technical lemma, which generalizes the result by \citet[Lemma 1.8]{santambrogio2015optimal} to the Gromov-Wasserstein setting. The proof appears in \Cref{app: santam1.8general}.

\begin{lemma}
    \label{lem: santam1.8general}
    Let $X$ and $Y$ be Polish spaces endowed with probability measures $\mu$ and $\nu$, respectively, and suppose $\gamma_n, \gamma \in \mathcal{C}(\mu,\nu)$. Let $a:X\times X \to \tilde{X}$ and $b:Y\times Y\to \tilde{Y}$ be measurable maps valued in  separable metric spaces $\tilde{X}$ and $\tilde{Y}$. Finally, let $c:\tilde{X}\times \tilde{Y}\to [0,\infty)$ be a continuous function such that $c(a,b)\leq f(a)+g(b)$ for some continuous maps $f$ and $g$ satisfying $\int (f\circ a)d(\mu\otimes \mu), \int (g\circ b)d(\nu\otimes \nu) <+\infty$. Then, $\gamma_n\to \gamma$ weakly implies
    \begin{align}
        &\int_{X\times Y}\int_{X\times Y}c(a(x,x'),b(y,y'))\gamma_n(dx \times dy) \gamma_n (dx' \times dy') \\
        & \qquad \qquad \qquad \qquad \to \int_{X\times Y}\int_{X\times Y}c(a(x,x'),b(y,y'))\gamma(dx \times dy) \gamma(dx' \times dy')
    \end{align}
   \begin{equation}
    \end{equation}
\end{lemma}

\begin{proof}of \Cref{thm: optimal coupling}.
    As mentioned earlier, it is enough to prove that $\mathcal{C}(\mu_X,\mu_Y)\ni \pi \mapsto \disp^Z(\pi)$ is continuous with respect to the topology of weak convergence. We can apply \Cref{lem: santam1.8general} by setting 
    \begin{equation} 
  \begin{gathered}
    \tilde{X}=\tilde{Y}=Z, \quad  a=\omega_X, \quad b=\omega_Y, \quad \mu=\mu_X, \quad \nu=\mu_Y,       \\
    c(a,b)=d(\omega_X,\omega_Y)^p, \quad  f(a)=2^{p-1}d_Z(a,z)^p, \quad  g(b)=2^{p-1}d_Z(b,z)^p,
  \end{gathered}
\end{equation}
for some fixed (but arbitrary) point $z\in Z$. The general estimate $(s+t)^p\leq 2^{p-1}(s^p+t^p)$, for $s,t\geq 0$, combined with the triangle inequality, shows that 
\[
d(\omega_X,\omega_Y)^p \leq 2^{p-1}d_Z(\omega_X,z)^p+2^{p-1}d_Z(\omega_Y,z)^p,
\]
so the assumption $c(a,b)\leq f(a)+g(b)$ is satisfied. The integrability of $f\circ a$ and $g\circ b$ is satisfied by the $L^p$ assumption on $a=\omega_X$ and $b=\omega_Y$, so all assumptions are clear. Thus, $\disp$ is continuous, and an optimal coupling exists for $1\leq p<\infty$. Since the $p=\infty$ case can be seen as the supremum of $\disp$ for $p<\infty$, we have lower semicontinuity, and an optimal coupling exists for $p=\infty$.
\end{proof}

\subsubsection{\texorpdfstring{$Z$}{Z}-Gromov-Wasserstein Distance Is a Metric}

It turns out that the $Z$-network Gromov-Wasserstein distances are only pseudometrics; that is, a zero distance does not imply that two $(Z,p)$-networks are equal. Similar to the argument by \citet{chowdhury2019gromov}, the following condition characterizes the equivalence class of the set of networks with zero distance, as we will prove below. Given functions $\omega:Y \times Y \to Z$ and $\phi:X \to Y$, we define the \define{pullback of $\omega$ by $\phi$} to be the function $\phi^\ast \omega:X \times X \to Z$ defined by 
\[
\phi^\ast \omega(x,x') = \omega(\phi(x),\phi(x')).
\]

\begin{definition}[Weak Isomorphism]\label{def:weak_isomorphism}
    A pair of $(Z,p)$-networks $X = (X,\omega_X,\mu_X)$ and $Y=(Y,\omega_Y,\mu_Y)$ is defined to be \define{weakly isomorphic} if there exists a $(Z,p)$-network $W = (W,\omega_W,\mu_W)$, and maps $\phi_X:W \to X$ and $\phi_Y:W \to Y$ such that
    \begin{itemize}
        \item $\phi_X$ and $\phi_Y$ are measure-preserving, and
        \item $\phi_X^\ast \omega_X = \phi_Y^\ast \omega_Y = \omega_W$ $\mu_W \otimes \mu_W$-almost everywhere. That is,
              \[
                  \omega_W(w,w') = \omega_X(\phi_X(w),\phi_X(w')) = \omega_Y(\phi_Y(w),\phi_Y(w'))
              \]
              for $\mu_W \otimes \mu_W$-almost every pair $(w,w') \in W \times W$.
    \end{itemize}
    If $X$ and $Y$ are weakly isomorphic, write $X \sim Y$.

    Let $\M^{Z,p}$ denote the \define{space of $(Z,p)$-networks}, and let $\MZp$ denote the \define{space of $(Z,p)$-networks considered up to weak isomorphism} (i.e., the quotient of $\M^{Z,p}$ by the weak isomorphism equivalence relation). 
\end{definition}

 In the following, we make no distinction between a $Z$-network $(X,\omega,\mu)$ and its equivalence class in $\MZp$, unless necessary.

We will now prove that the Gromov-Wasserstein distances define metrics up to weak isomorphism:
\begin{theorem}\label{thm:ZGW_is_a_metric}
    For any separable metric space $(Z,d_Z)$, $\dgwz$ induces a metric on $\MZp$.
\end{theorem}
 
For the Fused GW~\citep{vayer2020fused} and Fused Network GW~\citep{yang2023exploiting} distances, it was previously only shown that they satisfy ``relaxed" triangle inequalities, including an extra scaling term depending on the $q$ parameter---that is, an inequality of the form
\[
d_q(x,z) \leq 2^{q-1}(d_q(x,y) + d_q(y,z)),
\]
where we momentarily use $d_q$ as a placeholder for a generic metric depending on a parameter $q \in [1,\infty)$ (cf.~\citealp[Theorem 1]{vayer2020fused}).
With our slight reformulation of these distances (see \Cref{rem:qth_root} and \Cref{rem:qth_root_FGW}), we have the following corollary of \Cref{thm:ZGW_is_a_metric}. The result follows immediately by transporting the definition of weak isomorphism into the appropriate context and via \Cref{prop:FNGW_equals_ZGW} and \Cref{thm:ZGW_is_a_metric}. In particular, this result strengthens the previous results on relaxed triangle inequalities, promoting them to true triangle inequalities.
\begin{corollary}\label{cor:fusedGW_triangle}
    The Fused GW and Fused Network GW distances are metrics, up to weak isomorphism.
\end{corollary}
The proof of \Cref{thm:ZGW_is_a_metric} follows the standard strategy (cf.~\citealt{memoli2011}), for which we will need the following well-known lemma:
\begin{lemma}[Gluing Lemma, \citealt{villani2003topics}]
    Let $\mu_1,\mu_2,\mu_3$ be three probability measures, supported on Polish spaces $X_1,X_2,X_3$ respectively, and let $\pi_{12}\in \mathcal{C}(\mu_1,\mu_2), \pi_{23}\in \mathcal{C}(\mu_2,\mu_3)$ be two couplings. Then there exists a probability measure $\pi$ on $X_1\times X_2\times X_3$, with marginals $\pi_{12}$ on $X_1\times X_2$ and $\pi_{23}$ on $X_2\times X_3$.
\end{lemma}

\begin{proof}of~\Cref{thm:ZGW_is_a_metric}.
    First, it is clear that $\dgwz(X,Y)\geq 0$ for any $Z$-networks $X,Y$. For symmetry, the symmetry of $d$ proves that the mapping $f:X\times Y \to Y\times X, f(x,y)=(y,x)$ defines a bijection $\mathcal{C}(\mu_X,\mu_Y)\ni \pi \mapsto f_* \pi \in \mathcal{C}(\mu_Y,\mu_X)$ such that $\disp^Z(\pi) = \disp^Z(f_*\pi)$, and taking the infimum shows that $\dgwz(X,Y)=\dgwz(Y,X)$.  
    
    We will now prove the triangle inequality. We consider three $Z$-networks $(X_1,\omega_{1},\mu_1)$, $(X_2,\omega_{2}, \mu_2)$ and $(X_3, \omega_{3}, \mu_3)$ and take couplings $\pi_{12}\in \mathcal{C}(\mu_1,\mu_2)$ and $\pi_{23}\in \mathcal{C}(\mu_2,\mu_3)$ realizing $\dgwz(X_1,X_2)$ and $\dgwz(X_2,X_3)$, respectively (\Cref{thm: optimal coupling}). By the Gluing Lemma, there exists a probability measure $\pi$ on $X_1\times X_2 \times X_3$ with marginals $\pi_{12}$ and $\pi_{23}$. Denote the marginal of $\pi$ on $X_1\times X_3$ by $\pi_{13}$. Then, by definition, we have $\pi_{13}\in \mathcal{C}(\mu_1,\mu_3)$. Therefore, by Minkowski's inequality, we have
    \begin{align}
        2 \cdot \dgwz (X_1, X_3) & \leq \|d_Z\circ (\omega_1\times \omega_3)\|_{L^p(\pi_{13}\otimes \pi_{13})} \label{eqn:triangle_inequality_proof_1} \\ &= \|d_Z\circ (\omega_1\times \omega_3)\|_{L^p(\pi \otimes \pi)} \label{eqn:triangle_inequality_proof_2} \\ &\leq \|d_Z\circ (\omega_{1} \times \omega_{2}) + d_Z\circ (\omega_{2} \times \omega_{3})\|_{L^p(\pi\otimes \pi)} \label{eqn:triangle_inequality_proof_3} \\ &\leq \|d_Z\circ (\omega_{1} \times \omega_{2})\|_{L^p(\pi\otimes \pi)} + \|d_Z\circ (\omega_{2} \times \omega_{3})\|_{L^p(\pi\otimes \pi)} \label{eqn:triangle_inequality_proof_4} \\ &=\|d_Z\circ (\omega_{1} \times \omega_{2})\|_{L^p(\pi_{12}\otimes \pi_{12})} + \|d_Z\circ (\omega_{2} \times \omega_{3})\|_{L^p(\pi_{23}\otimes \pi_{23})} \label{eqn:triangle_inequality_proof_5} \\ &= \disp^Z(\pi_{12}) + \disp^Z(\pi_{23})\\
        &= 2 \cdot \big(\dgwz(X_1,X_2) + \dgwz(X_2,X_3)\big) \label{eqn:triangle_inequality_proof_6},
    \end{align}
    where we have used suboptimality of $\pi_{13}$ in \eqref{eqn:triangle_inequality_proof_1}, marginal conditions on $\pi$ in \eqref{eqn:triangle_inequality_proof_2} and \eqref{eqn:triangle_inequality_proof_5},  the triangle inequality for $d_Z$ in \eqref{eqn:triangle_inequality_proof_3}, Minkowski's inequality in \eqref{eqn:triangle_inequality_proof_4}, and the optimality of $\pi_{12}$ and $\pi_{23}$ in \eqref{eqn:triangle_inequality_proof_6}.
    
    We finally prove that, for $Z$-networks $X$ and $Y$, $\dgwz(X,Y)=0$ if and only if $X$ and $Y$ are weakly isomorphic. We will restrict our attention to the case $1\leq p <\infty$ since we can prove the $p=\infty$ case by simply replacing the integrals with essential suprema. Suppose that $X$ and $Y$ are weakly isomorphic. Then, there exists a $Z$-network $(W,\omega_W, \mu_W)$ and measure-preserving maps $\phi_X:W\to X$ and $\phi_Y:W\to Y$ such that $\phi_X^*\omega_X = \phi_Y^* \omega_Y = \omega_W$ $\mu_W\otimes \mu_W$-almost everywhere. By definition of $\phi_X$ and $\phi_Y$,  $\pi_{XY}=(\phi_X, \phi_Y)_*\mu_W$ belongs to $\mathcal{C}(\mu_X,\mu_Y)$, and
    \begin{align}
        \disp^Z(\pi)^p & = \int_{X\times Y}\int_{X\times Y}d_Z(\omega_X(x,x'),\omega_Y(y,y'))^p\pi(dx\times dy)\pi(dx'\times dy') \\ &= \int_{W}\int_{W}d_Z(\omega_X(\phi_X(dw), \phi_X(dw')), \omega_Y(\phi_Y(dw),\phi_Y(dw')) )^p\mu_W(dw)\mu_W(dw') \\ &=\int_{W}\int_{W}d_Z(\phi_X^*\omega_X(w, w'), \phi_Y^*\omega_Y(w, w') )^p\mu_W(dw)\mu_W(dw') \\ &=\int_{W}\int_{W}d_Z(\omega_W(w,w'), \omega_W(w, w') )^p\mu_W(dw)\mu_W(dw')=0.
    \end{align}
    Therefore, $\dgwz (X,Y) = 0$. To prove the other direction, suppose that $\dgwz(X,Y)=0$. By \Cref{thm: optimal coupling}, there exists $\pi \in \mathcal{C}(X\times Y)$ such that $\disp^Z(\pi)=0$. Now, define measure-preserving maps $\phi_X:X\times Y\to X, \phi_Y: X\times Y \to Y$ as the projection onto $X$ and $Y$, respectively, and define a $Z$-network by $(W,\omega_W,\mu_W)=(X\times Y, \phi_X^* \omega_X, \pi)$. Now, we have
    \begin{align}
        0=\disp^Z(\pi)^p & = \int_{X\times Y}\int_{X\times Y}d_Z(\omega_X(x,x'),\omega_Y(y,y'))^p\pi(dx\times dy)\pi(dx'\times dy') \\ &= \int_{W}\int_{W}d_Z(\phi_X^*\omega_X(w,w'),\phi_Y^*\omega_Y(w,w'))^p\pi(dw)\pi(dw')
    \end{align}
    Since $d_Z\geq 0$, we have $d_Z(\phi_X^*\omega_X(w,w'),\phi_Y^*\omega_Y(w,w'))=0$ $\pi\otimes \pi$-almost everywhere, or equivalently, $\phi_W = \phi_X^*\omega_X=\phi_Y^*\omega_Y$ $\pi\otimes \pi$-almost everywhere. Therefore, $X$ and $Y$ are weakly isomorphic.
\end{proof}

\subsubsection{\texorpdfstring{$Z$-Gromov-Wasserstein Space as a Quotient of the $L^p$ Space}{Z-Gromov-Wasserstein Space as a Quotient of the Lp Space}} 

One of the classical results in metric measure geometry is that every metric measure space $(X,d_X,\mu_X)$ admits a \define{parameter}, i.e., a surjective Borel measurable map $\rho:[0,1]\to X$ such that $\rho_*\mathscr{L}=\mu_X$; see works by  \citet[Section 3$\frac{1}{2}$.3]{Gromov2006Metric} or \citet[Lemma 4.2]{Shioya2016Metric}.

This result was applied by \citet{chowdhury2019gromov} and \citet{sturm2023space} to obtain convenient representations of ($\R$-valued) measure networks.  Here, we consider a generalization of this construction to the setting of $Z$-networks. The proof of the following proposition is obvious from its formulation.

\begin{proposition}[Parametrization]
    \label{rem: interval representation} Any $(Z,p)$-network $X=(X,\omega_X, \mu_X)$ is weakly isomorphic to a $(Z,p)$-network $([0,1], \rho^*\omega_X, \mathscr{L})$ where $\mathscr{L}$ is the Lebesgue measure on $[0,1]$, $\rho:[0,1]\to X$ is a parameter of $X$, and $\rho^*\omega_X:[0,1]^2\to Z$ is defined by $\rho^*\omega_X(s,t)=\omega_X(\rho(s),\rho(t))$.
\end{proposition}

An interesting observation is that, given a $Z$-network $([0,1],\omega,\mathscr{L})$ already defined over $[0,1]$, by considering parameters $\rho:[0,1]\to[0,1]$, we can construct a family of $Z$-networks $([0,1],\rho^*\omega,\mathscr{L})$, all of which are weakly isomorphic to the original $Z$-network. Following this idea, \citet{sturm2023space} established an instructive representation of the standard GW distance; here, we will generalize their construction to the $Z$-GW distance.

\begin{definition}[Invariant Transforms, \citealt{sturm2023space}]
    We say that a Borel measurable map $\phi:[0,1]\to [0,1]$ is \define{$\mathscr{L}$-invariant} if $\phi_* \mathscr{L}=\mathscr{L}$. The $\mathscr{L}$-invariant maps form a semigroup via function composition, which we denote by $\textnormal{Inv}([0,1],\mathscr{L})$. 
\end{definition}

For $p\in [1,\infty]$, $\textnormal{Inv}([0,1],\mathscr{L})$ acts on the space of kernels $L^p([0,1]^2, \mathscr{L}^2;Z)$ via the pullback $\omega \mapsto \phi^* \omega$, and the action defines the following equivalence relation on the $L^p$ space , which closely relates to weak isomorphism:
\begin{equation}
    \omega \simeq \omega' \Leftrightarrow \exists \phi, \psi\in\textnormal{Inv}([0,1],\mathscr{L}) \textrm{ s.t. }\phi^*\omega=  \psi^*\omega'.
\end{equation}
Denote the equivalence class of $\omega \in L^p([0,1]^2, \mathscr{L}^2;Z)$ by $[\omega]$. The quotient space 
\[
\mathbb{L}^p \coloneqq L^p([0,1]^2, \mathscr{L}^2;Z)/\textrm{Inv}([0,1], \mathscr{L})
\]
admits a metric
\begin{equation}
    \mathbb{D}_p([\omega], [\omega']) \coloneqq \frac{1}{2}\inf\left\{D_p(\phi^* \omega, \psi^* \omega')|\phi,\psi \in \textrm{Inv}([0,1], \mathscr{L})\right\},
\end{equation}
where $D_p$ is the $L^p$ distance (see \Cref{def:metric_lp_spaces}) and this family of metric spaces provides a representation of the $Z$-GW space:
\begin{theorem}
    \label{thm:quotient_rep_Lp}
    For $p\in [1,\infty]$, $(\mathbb{L}^p, \mathbb{D}_p)$ is isometric to $(\MZp, \dgwz)$ by 
    \begin{equation}
    \Theta_p: \mathbb{L}^p\ni [\omega] \mapsto [([0,1],\omega, \mathscr{L})] \in \MZp
    \end{equation}
    where $[([0,1],\omega, \mathscr{L})]$ denotes the equivalence class of a $(Z,p)$-network $([0,1],\omega, \mathscr{L})$ by weak isomorphism.
\end{theorem}
\begin{proof}
    As the proof by \citet[Theorem 5.10]{sturm2023space} does not rely on the properties of the target metric space, trivial modifications allow us to apply the proof to our case.
\end{proof}

\subsection{Metric and Topological Properties}

Having shown that the $Z$-GW distance is a metric, we now establish some of the properties of its induced topology. In particular, under mild assumptions, we will prove that:
\begin{itemize}
    \item It is \emph{separable}, i.e., it contains a countable, dense subset. This is crucial for approximation results, in that we can frequently pass to a relatively simple subspace.
    \item It is \emph{complete}, i.e., its Cauchy sequences always converge. Completeness and separability together make the space of $Z$-networks a \emph{Polish space}; this is generally considered to be the minimum requirement for the theory of measures over a space to be well-behaved (see \Cref{sec:basic_terminology}).
    \item It is \emph{path-connected} and, in fact, \emph{contractible}. The existence of paths joining points in the space is required for the application of statistical methods. For example, it allows interpolations between points, which is a fundamental ingredient for the computation of geometric statistics such as Fr\'{e}chet means. Contractibility (that the space of $Z$-networks is homotopy equivalent to a point) tells us that the space is quite simple, from a topological perspective.
\end{itemize}
We also show that the $Z$-Gromov-Wasserstein space is \emph{geodesic}, provided that $Z$ is. Recall that a metric space $(X,d_X)$ is \define{geodesic} if, for any pair of points $x,x' \in X$, there exists a \define{geodesic path} joining them; this is a path $\gamma:[0,1] \to X$ such that, for all $s,t \in [0,1]$,
\[
d_X(\gamma(s),\gamma(t)) = |t-s| d_X(x,x')
\]
(see, e.g., \citet{burago2022course} for details). In fact, to check that a given curve $\gamma$ is a geodesic path, it suffices to show that $d_X(\gamma(s),\gamma(t)) \leq |t-s| d_X(x,x')$ holds for all $s,t \in [0,1]$ \citep[Lemma 1.3]{chowdhury2018explicit}. In a geodesic space, points are not only always connected by paths, but are connected by paths which behave like `straight lines' in the space.

\subsubsection{Separability}
We first show that the $Z$-GW distance inherits separability from its target space $Z$. We will use the following lemma. 

\begin{lemma}
    \label{lem: lp comparison} Let $X_0=(X,\omega_0, \mu), X_1=(X,\omega_1,\mu)$ be $(Z,p)$-networks with the same underlying Polish space and measure but different kernels. We have 
    \begin{equation}
        \dgwz(X_0, X_1) \leq \frac{1}{2} D_p(\omega_0, \omega_1)
    \end{equation}
    Here, $D_p$ denotes the $L^p$ distance of \Cref{def:metric_lp_spaces}.
\end{lemma}
\begin{proof}\belowdisplayskip=-12pt
    Let $\Delta: X\to X\times X, \Delta(x)=(x,x)$ be the mapping to the diagonal. We note that $\Delta_*\mu$ defines a coupling between $\mu$ and itself, so substituting $\pi=\Delta_*\mu$ to the distortion functional, we obtain, for $p\in [1,\infty)$ (the $p = \infty$ case is similar),
    \begin{align}
        \dgwz(X_0,X_1) \leq \frac{1}{2} \disp^Z(\pi)&=\frac{1}{2}\left(\int_{X}\int_{X}d_Z(\omega_0(x,x'),\omega_1(x,x'))^p\mu(dx)\mu(dx')\right)^{1/p} \\ &= \frac{1}{2}D_p(\omega_0,\omega_1).
    \end{align}
\end{proof}

\begin{proposition}
    \label{prop: separability}
    $\MZp$ is separable.
\end{proposition}
\begin{proof}
     Without loss of generality (by \Cref{rem: interval representation}), we consider $Z$-networks of the form $([0,1],\omega, \mathscr{L})$ where $\mathscr{L}$ is the Lebesgue measure. Since $L^p([0,1]\times [0,1], \mathscr{L}\otimes \mathscr{L};Z)$ is separable by \Cref{prop: lp complete_separable}, we can take a countable dense subset $\{\omega_i\}_{i=1}^\infty$ of $L^p([0,1]\times [0,1], \mathscr{L}\otimes \mathscr{L};Z)$. Now, for any $Z$-network $X=([0,1], \omega, \mathscr{L})$ and any $\epsilon>0$, there exists $\omega_i$ such that $D_p(\omega, \omega_i) < \epsilon$. By constructing a $Z$-network by $X_i=([0,1], \omega_i, \mathscr{L})$, and by \Cref{lem: lp comparison}, we see that $\dgwz(X,X_i)<\epsilon$. Therefore, $\{X_i\}_{i=1}^{\infty}$ is dense in $\MZp$.
\end{proof}
Although the proof above is sufficient to prove the separability, the argument remains high-level and does not provide an explicit countable dense subset in $\MZp$. To construct such a subset, we introduce the following object.

\begin{definition}[$n$-Point $Z$-Network]
    For each $n\in \mathbb{N}$, an \define{$n$-point $Z$-network} is the equivalence class of a $Z$-network $(X,\omega, \mu)$ by weak isomorphism,  where $X=\{1,\ldots, n\}$, $\omega:X\times X\to Z$ is any function and $\mu= \frac{1}{n}\sum_{i=1}^{n}\delta_{i}$. Denote the set of $n$-point $(Z,p)$-networks by $\M^{Z}_n$.
\end{definition}

We note that $n$-point $Z$-networks are $(Z,p)$-networks for any $p\in[1,\infty]$ because the kernels are supported on finite sets. That is, we have $\M^{Z}_n \subset \MZp$ for any $p\in [1,\infty]$. \citet{sturm2023space} proved the density of $n$-point $Z$-networks in the special case where he only considers symmetric $\mathbb{R}$-valued kernels $\omega$. However, the proof can be generalized to $Z$-networks.

\begin{proposition}\label{prop:dense_subset}
    The countable set $\bigcup_{n}\M^{Z}_n$ is dense in $\MZp$.
\end{proposition}
\begin{proof}
    We first take an arbitrary $(Z,p)$-network. By parameterization (\Cref{rem: interval representation}), it is weakly isomorphic to a $(Z,p)$-network of the form $X=([0,1],\omega, \mathscr{L})$. The kernel $\omega$ belongs to $L^p([0,1]^2, \mathscr{L}^2; Z)$, so by \Cref{prop: lp complete_separable}, it can be approximated by a piecewise constant function $\omega_n$ that is constant on a grid of step size $1/n$ with $n$ small enough. That is, we can partition $[0,1]$ into $n$ pieces $R_i = [i/n, (i+1)/n)$ for $i=0,\ldots, n-2$, $R_{n-1}=[(n-1)/n,n]$ and take $\omega_n$ to be constant on each $R_i\times R_j$. The $Z$-network $X_n=([0,1],\omega_n, \mathscr{L})$ is weakly isomorphic to $(\{1,\cdots, n\}, \Omega_n, \frac{1}{n}\sum_{i=1}^{n}\delta_i)$ where $\Omega_n$ is the $n\times n$, $Z$-valued matrix with $\Omega_n(i,j) = \omega_n(i/n, j/n)$ because we can construct a measure preserving map $\phi: X\to \{1,\cdots , n\}$ by $\phi(x) = i$ whenever $x\in R_i$, and this map satisfies $\phi^*\Omega_n = \omega_n$ by definition. Finally, the claim follows by applying \Cref{lem: lp comparison} to $X$ and $X_n$.
\end{proof}

\subsubsection{Completeness}
The $Z$-GW space also inherits completeness from its target space $Z$. Moreover, the opposite direction is also true.

\begin{theorem}\label{thm:complete}
    $\MZp$ is complete if and only if $Z$ is complete.
\end{theorem}
\begin{proof}
    Suppose $Z$ is complete. Then, the proof that $\dgwz$ is complete follows exactly as in the case $Z=\mathbb{R}$ \citep[Theorem 5.8]{sturm2023space}, since $L^p(X\times X, \mu_X \otimes \mu_X; Z)$ is complete (\Cref{prop: lp complete_separable}). This proves one direction of the theorem.
    Let us now consider the converse. We will generalize the proof for the Wasserstein space \citep{pinelis_complete} to our setting. Suppose that $\dgwz$ is complete. Let $\{z_n\}$ be a Cauchy sequence in $Z$. We will show that $\{z_n\}$ converges to some $z\in Z$. For each $n$, let $X_n = (\{0\}, \omega_n, \delta_{0})$ where $\omega_n:\{0\}\times \{0\}\to Z$ is the function $\omega_n(0,0)=z_n$. Then, $\{X_n\}$ is a Cauchy sequence in $\MZp$ with respect to $\dgwz$ because $\dgwz(X_n, X_m)=\frac{1}{2}d_Z(z_n,z_m)$. By completeness, there exists $X=(X,\omega,\mu)\in \MZp$ such that $\dgwz(X_n,X)\to 0$. Since the only coupling with a Dirac measure is the product measure, we can explicitly calculate $\dgwz(X_n,X)$ as
    \begin{align}
        2^p\dgwz(X_n,X)^p & = \int_{X}\int_{X}d_Z(z_n,\omega(x,x'))^p\mu(dx)\mu(dx')  =\int_{Z}d_Z(z_n,z)^p\omega_*(\mu \otimes \mu)(dz)
    \end{align}
    for $p<\infty$. For $p=\infty$, we have 
    \begin{equation}
        \dgwinf(X_n, X) = \frac{1}{2}\|d(z_n, z)\|_{L^{\infty}(\omega_*(\mu\otimes\mu))}
    \end{equation}
    
    For brevity, we denote $\omega_\ast(\mu\otimes \mu)$ by $\mu_{\omega}$. We note that $\mu_{\omega}$ is a probability measure on $Z$ since $\mu\otimes \mu$ is a probability measure on $X\times X$. Therefore, if $p<\infty$, Markov's inequality is applicable to show that
    \begin{equation}
        \epsilon^p(1-\mu_{\omega}(B_{\epsilon}(z_n))) = \epsilon^p \mu_{\omega}(\{z\in Z: d_Z(z_n,z)\geq \epsilon\}) \leq \int_{Z}d_Z(z_n,z)^p\mu_{\omega}(dz)
    \end{equation}
    Moreover, if $p=\infty$, by the monotonicity of $L^p$ norms,
    \begin{equation}
        \epsilon(1-\mu_{\omega}(B_{\epsilon}(z_n)) = \epsilon\mu_{\omega}(\{z\in Z: d_Z(z_n,z)\geq \epsilon\}) \leq \|d_Z(z_n,z)\|_{L^1(\mu_\omega)} \leq \|d_Z(z_n,z)\|_{L^{\infty}(\mu_\omega)},
    \end{equation}
    which shows that $\mu_{\omega}(B_{\epsilon}(z_n))\to 1$ as $n\to \infty$ for any $\epsilon>0$. Here, $B_{\epsilon}(z_n)=\{z\in Z :d_Z(z_n,z)<\epsilon\}$. 
    
    Intuitively, the above means that the mass of $\mu_{\omega}$ is concentrated around $z_n$ as $n\to \infty$, or in other words, $z_n$ gets closer and closer to the support point of $\mu_{\omega}$. Indeed, for any $z\in \textrm{supp}(\mu_{\omega})$, we have $\mu_{\omega}(B_{\epsilon}(z))>0$ for any $\epsilon>0$ by definition. This means that, if $n$ is large enough, $B_{\epsilon}(z)$ and $B_{\epsilon}(z_n)$ should have a nonempty intersection for any $\epsilon>0$ because $\mu_{\omega}(B_{\epsilon}(z))>0$ is a constant and $\mu_{\omega}(B_{\epsilon}(z_n))\to 1$, so the sum of two measures will be larger than $1$ at some point. Thus, for any $\epsilon>0$, there exists a natural number $N$ such that if $n\geq N$, we can take $v\in \textrm{supp}(\mu_{\omega})\cap B_{\epsilon}(z_n)$ so that $d_Z(z,z_n)\leq d_Z(z,v)+d_Z(v,z_n)<2\epsilon$. Since $\epsilon>0$ is arbitrary, this shows that $z_n\to z$ as $n\to \infty$. As any Cauchy sequence converges, we have shown that $Z$ is complete.
\end{proof}

\subsubsection{Path-Connectedness}

One of the distinct properties of $\MZp$ is that it is always contractible, regardless of the topology of $Z$. We prove this below, but first show that $\MZp$ is always path-connected. Intuitively, the result holds because a $Z$-network is equipped with a measure, which is a continuous object taking values on the real line. This intuition plays a fundamental role in the following proof. An important feature of the proof is that it provides an explicit formula for a continuous path between any pair of $Z$-networks. This formula will be used later in the proof of contractibility (see \Cref{lem:one_point_Holder}).

\begin{proposition}
    \label{prop:path-connected}
    For any space $Z$, $\MZp$ is path-connected for all $p\in[1,\infty)$.
\end{proposition}

The $p=\infty$ case is more subtle, and is not completely resolved, as we describe below.

\begin{proof}
 For any $Z$-networks $(X,\omega_X, \mu_X), (Y,\omega_Y,\mu_Y)$, define a path of $Z$-networks by 
 \begin{equation}\label{eqn:explicit_path}
 X_t = (X\amalg Y, \omega_X \amalg \omega_Y, (1-t)\mu_X+t\mu_Y) \quad \mbox{for} \quad t\in [0,1]. 
 \end{equation}
 Here, $\omega_X \amalg \omega_Y: X\amalg Y\to Z$ is any function (independent of $t$) that satisfies $\omega_X \amalg \omega_Y(x,x') = \omega_X(x,x')$ for $x,x' \in X$ and $\omega_X \amalg \omega_Y(y,y') = \omega_Y(y,y')$ for $y,y' \in Y$, and the measure is defined by 
 \[
 ((1-t)\mu_X+t\mu_Y)(A) = (1-t)\mu_X(A\cap X) + t\mu_Y(A\cap Y)
 \]
 for any measurable set $A\subset X\amalg Y$. 
 
 First, we note that $X_0$ and $X_1$ are weakly isomorphic to $X$ and $Y$, respectively. For $X_0$ and $X$, this can be seen by pulling back $X\amalg Y$ to $X$ by the inclusion map. In other words, for two $Z$-networks $X=(X,\omega_X,\mu_X)$ and $X_0=(X\amalg Y, \omega_X\amalg \omega_Y, \mu_X)$, we consider $W=X$, the identity map $\phi_X:W\to X$ and the inclusion map $\phi_{X_0}:X\to X\amalg Y$. The triple $W,\phi_X,\phi_{X_0}$ satisfies the definition of weak isomorphism. The $X_1$ case is similar.
 
 We will now prove that the path $X_t$ is continuous with respect to the GW distance. To see this, take any coupling $\pi$ between $X,Y$ and define a coupling $\pi_{s,t}$ between $X_s$ and $X_t$ for $s<t$ by 
 \[
 \pi_{s,t}=(1-t)\Delta^X_*\mu_X+(t-s)\pi+s\Delta^Y_*\mu_Y
 \]
 where 
 \[
 \Delta^X:X\to (X\amalg Y)^2 \quad \mbox{and} \quad \Delta^Y: Y\to (X\amalg Y)^2
 \]
 are mappings into the diagonal of $X$ and $Y$, respectively. Moreover, $\pi$ here as a measure on $(X\amalg Y)^2$ is defined as $\pi(A\cap (X\times Y))$ for measurable $A\subset (X\amalg Y)^2$. The coupling $\pi_{s,t}$ represents a transport plan where we keep $1-t$ and $s$ units of mass at $X$ and $Y$, respectively, and send $t-s$ units from $X$ to $Y$. By definition, we have
    \begin{align}
        2^p\dgwz(X_s, X_t)^p \leq & (1-t)^2 I_{X X} + (t-s)^2I_{\pi\pi} + t^2I_{YY}+2(1-t)(t-s)I_{X \pi} \\ 
        & \hspace{2in} +2(t-s)sI_{\pi Y}+ 2s(1-t)I_{Y X}.
    \end{align}
    Here, $I_{\mu\nu}$, where $\mu,\nu \in \{X,Y,\pi\}$ is the integral
    \begin{equation}
        \int_{(X\amalg Y)^2}\int_{(X\amalg Y)^2}d_Z(\omega_X\amalg \omega_Y(u,u'), \omega_X \amalg \omega_Y(v,v'))^p\mu(du\times dv)\nu(du'\times dv'),
    \end{equation}
    where $X$ and $Y$ are abusively being used as shorthand for $\Delta^X_*\mu_X$ and $\Delta^Y_*\mu_Y$, respectively. We note that $I_{XX}, I_{YX}$ and $I_{YY}$ vanish since the integrand is identically zero by the definition of $\Delta_*^X\mu_X$ and $\Delta_*^Y\mu_Y$. Therefore,
    \begin{align}
        2^p\dgwz(X_s, X_t)^p & \leq (t-s)^2I_{\pi\pi} +2(1-t)(t-s)I_{X \pi} +2(t-s)sI_{\pi Y} \\ &\leq (t-s)(I_{\pi\pi} +2I_{X\pi} + 2 I_{\pi Y}) \label{eqn:proof_of_continuity_paths}
    \end{align}
    Since $I_{\pi\pi}, I_{X\pi}$ and $I_{\pi Y}$ are independent of time, this proves that our path $X_t$ is $1/p$-Hölder continuous. Therefore, $\MZp$ is path-connected.
\end{proof}

Let us now discuss the path-connectivity properties of $\mathcal{M}^{Z,\infty}_\sim$. 
To do so, it will be useful to introduce an invariant of a $Z$-network $X = (X,\omega_X,\mu_X)$. For a fixed point $z \in Z$, we define the \define{$p$-size of $X$, relative to $z$} to be 
\begin{equation}\label{eqn:size_function}
\mathrm{size}_{p,z}(X) \coloneqq \|d_Z(\omega_X(\cdot,\cdot),z)\|_{L^p(\mu_X \otimes \mu_X)}
\end{equation}
(see also~\Cref{def:invariants} below). It is a fact that, for any $Z$-networks $X$ and $Y$,
\begin{equation}\label{eqn:size_bound}
\dgwz(X,Y) \geq \frac{1}{2}|\mathrm{size}_{p,z}(X) - \mathrm{size}_{p,z}(Y)|.
\end{equation}
Indeed, we prove this below in~\Cref{thm:lower_bounds}, in the context of several lower bounds on the $Z$-GW distance.

We now show that \Cref{prop:path-connected} fails in the $p=\infty$ case. Intuitively, the proof strategy does not apply because it is based on manipulating ``weights" in the measures, whereas $\mathrm{GW}^Z_\infty$ is insensitive to these weights and only optimizes a quantity which depends on supports of couplings.

\begin{example}\label{ex:no_path_continuity}
    Let $Z = \{0,1\}\subset \R$, with the restriction of Euclidean distance. We claim that $\mathcal{M}^{Z,\infty}_\sim$ is not path connected. To see this, observe that the size function \eqref{eqn:size_function} has the property that
    \[
    \mathrm{size}_{\infty,0}(X) \in \{0,1\}
    \]
    for all $X \in \mathcal{M}^{Z,\infty}_\sim$. By \Cref{thm:lower_bounds} (or \eqref{eqn:size_bound}), the function $\mathrm{size}_{\infty,0}:\mathcal{M}^{Z,\infty}_\sim \to \R$ is Lipschitz continuous. This means that, if $X_t$ is a continuous path in $\mathcal{M}^{Z,\infty}_\sim$, then the composition 
    \[
    t \mapsto \mathrm{size}_{\infty,0}(X_t):[0,1] \to \{0,1\}
    \]
    must be continuous, and therefore constant. It follows that there is no path joining a $(Z,\infty)$-network of size $0$ to one of size $1$. For a specific example, there is no continuous path from 
    \[
    X=(\{0\},\omega_X(0,0)=0,\delta_0) \quad \mbox{to} \quad Y=(\{1\},\omega_Y(1,1)=1, \delta_1)
    \]
    in $\mathcal{M}^{Z,\infty}_\sim$. 
\end{example}

The idea of the justification in \Cref{ex:no_path_continuity} can be extended to show that path-continuity of $Z$ is a necessary condition for path-continuity of $\mathcal{M}_\sim^{Z,\infty}$ (we omit the details of this extension here). However, we were unable to determine whether this condition is sufficient. On the other hand, if $Z$ is geodesic then so is $\mathcal{M}_\sim^{Z,\infty}$, as we show below in \Cref{thm:geodesic}. This observation illustrates  the subtlety of the following question:

\begin{question}\label{question:path_connectivity}
    Does path-connectivity of $Z$ imply path-connectivity of $\mathcal{M}_\sim^{Z,\infty}$?
\end{question}

\subsubsection{Contractibility} 

We now take the path-connectivity result (\Cref{prop:path-connected}) a step further and show that $\MZp$ is always contractible. 

\begin{theorem}
    \label{thm:contractible}
    For any space $Z$, $\MZp$ is contractible for all $p \in [1,\infty)$.
\end{theorem}

The proof will use the following lemma, which specializes some features of the proof of \Cref{prop:path-connected} to the case where one of the $Z$-networks is a one-point space. The result uses the size functions introduced in \eqref{eqn:size_function}. The proof of the lemma is provided in \Cref{app:one_point_Holder}.

\begin{lemma}\label{lem:one_point_Holder}
    Let $X = (X,\omega_X,\mu_X)$ be an arbitrary $Z$-network, and let $Y = (\{\star\},\omega_Y,\delta_\star)$ be a one-point $Z$-network, where $\omega_Y(\star,\star) = z$ for some fixed $z \in Z$, and $\delta_\star$ denotes the Dirac measure. Let $X_t$ denote the path defined in \eqref{eqn:explicit_path} between $X$ and $Y$ (considered up to weak isomorphism), where we specifically define 
    \[
    \omega_X \amalg \omega_Y(u,v) := \left\{\begin{array}{lr}
    \omega_X(u,v) & \mbox{if } u,v \in X \\
    z & \mbox{otherwise.}
    \end{array}\right.
    \]
    Then the H\"{o}lder estimate \eqref{eqn:proof_of_continuity_paths} simplifies to
    \begin{equation}\label{eqn:Holder_estimate_one_point}
    \dgwz(X_s,X_t) \leq \frac{(3 |t-s|)^{1/p}}{2} \cdot \mathrm{size}_{p,z}(X).
    \end{equation}
\end{lemma}

We now proceed with the proof of \Cref{thm:contractible}. The strategy is to define an explicit contraction to a one-point space which follows paths as defined in the proof of \Cref{prop:path-connected}. Namely, fix an arbitrary point $z \in Z$ and, for any $Z$-network $X = (X,\omega_X,\mu_X)$ and $t \in [0,1]$, let $X_t$ denote the $Z$-network 
 \begin{equation}\label{eqn:contractible}
 X_t = (X \amalg \{\star\}, \hat{\omega}_X, (1-t) \mu_X + t \delta_\star),
 \end{equation}
 where $\star$ is an abstract point, 
 \[
 \hat{\omega}_X(x,x') = \left\{\begin{array}{lr}
 \omega_X(x,x') & x,x' \in X \\
 z & \mbox{otherwise,}\end{array}\right.
 \]
 and $\delta_\star$ is the Dirac measure. This is then the path from $X$ to a one-point space, as in \Cref{prop:path-connected} and \Cref{lem:one_point_Holder}, where we are using the notation $\hat{\omega}_X$ rather than $\omega_X \amalg \omega_Y$ to condense notation. This will be used to construct the contraction.

\begin{proof}of \Cref{thm:contractible}.
  We define a map $\Phi:\MZp \times [0,1] \to \MZp$ by 
 \[
 \Phi\big([X],t) = [X_t],
 \]
 where we are using $[X]$ to denote the weak isomorphism class of a $Z$-network $X$, and $X_t$ is as in \eqref{eqn:contractible}. Arguments similar to those used in the proof of \Cref{prop:path-connected} show that $\Phi$ is well defined, and that, for any $Z$-network $X$, $X_0 \sim X$, and $X_1$ is weakly isomorphic to the $Z$-network 
 \[
 (\{\star\},\omega_{\{\star\}}, \delta_\star), \quad \mbox{with} \quad \omega_{\{\star\}}(\star,\star) = z,
 \]
 so that $\Phi(\cdot,1)$ is a constant map. It remains to show that $\Phi$ is continuous. This will be done in two steps: {\bf Step 1.} We show that the component function $\Phi(\cdot,t)$ is Lipschitz continuous for each fixed $t \in [0,1]$. {\bf Step 2.} We then combine this with various estimates to derive continuity in general.
 
 \smallskip
 \noindent {\it Step 1 (Lipschitz Continuity for Fixed $t$).} Following the plan described above, fix $t \in [0,1]$ and consider the restricted map
 \[
 \Phi(\cdot,t):\MZp \to \MZp.
 \]
 Toward establishing its continuity, let $X$ and $Y$ be $Z$-networks and choose an optimal coupling $\pi \in \mathcal{C}(\mu_X,\mu_Y)$. Let $X_t$ and $Y_t$ be as in \eqref{eqn:contractible}. We extend $\pi$ to a measure on 
 \begin{equation}\label{eqn:product_space_decomposition}
 W = (X \amalg \{\star\}) \times (Y \amalg \{\star\}) = (X \times Y) \amalg (X \times \{\star\}) \amalg (\{\star\} \times Y) \amalg \{(\star,\star)\}
 \end{equation}
 as $\overline{\pi} = (1-t) \iota_\ast \pi + t \delta_{(\star,\star)}$, where $\iota:X \times Y \hookrightarrow W$ denotes the inclusion map. Observe that, amongst the terms in the decomposition \eqref{eqn:product_space_decomposition}, $\overline{\pi}$ is supported only on $X \times Y$ and $\{(\star,\star)\}$.
 The $Z$-GW distance between $X_t$ and $Y_t$ is bounded above by the cost of the coupling $\overline{\pi}$; in the case that $p < \infty$ (with the $p=\infty$ case being similar), this implies
{\small
 \begin{align}
     2^p\dgwz(X_t,Y_t)^p &\leq \iint_{W \times W} d_Z(\hat{\omega}_X(x,x'),\hat{\omega}_Y(y,y'))^p \overline{\pi}(dx \times dy) \overline{\pi}(dx' \times dy') \\
     &= \left(\iint_{(X \times Y)^2} + \iint_{\{(\star,\star)\}^2} \right)d_Z(\hat{\omega}_X(x,x'),\hat{\omega}_Y(y,y'))^p \overline{\pi}(dx \times dy) \overline{\pi}(dx' \times dy')  \label{eqn:contractible_proof_1} \\
     &= (1-t)^2 \iint_{(X \times Y)^2} d_Z(\omega_X(x,x'),\omega_Y(y,y'))^p \pi(dx \times dy) \pi(dx' \times dy') \label{eqn:contractible_proof_2} \\
     &= (1-t)^2 2^p\dgwz(X,Y)^p, \label{eqn:contractible_proof_3}
 \end{align}}with the various equalities justified as follows:  \eqref{eqn:contractible_proof_1} uses the observation above on the support of $\overline{\pi}$, together with the fact that 
 \[
 d_Z(\hat{\omega}_X(x,\star),\hat{\omega}_Y(y,\star)) = d_Z(\hat{\omega}_X(\star,x),\hat{\omega}_Y(\star,y)) = 0, \qquad \forall \;(x,y) \in X \times Y,
 \] 
 so that the remaining ``cross-term" integrals $\iint_{(X\times Y) \times \{(\star,\star)\}}$ and $\iint_{\{(\star,\star)\} \times (X\times Y)}$ vanish; \eqref{eqn:contractible_proof_2} follows by the definitions of $\hat{\omega}_X$, $\hat{\omega}_Y$ and $\overline{\pi}$; finally,  \eqref{eqn:contractible_proof_3} is given by the optimality of $\pi$. This shows that $\Phi(\cdot,t)$ is Lipschitz continuous.

 \smallskip
 \noindent {\it Step 2 (Continuity in General).}
 We now show that the map $\Phi$ is continuous. Fix a $Z$-network $Y$ and $t \in [0,1]$, and let $\epsilon > 0$. For the moment, let us assume that $\mathrm{size}_{p,z}(Y) > 0$ and $t < 1$ (the remaining special case will be addressed later), so that we may assume without loss of generality (for technical reasons) that
 \begin{equation}\label{eqn:epsilon_assumption}
 \epsilon < \frac{3}{2} \mathrm{size}_{p,z}(Y) (1-t)^{2/p}, \quad \mbox{or} \quad \frac{2\epsilon}{3\,\mathrm{size}_{p,z}(Y)(1-t)^{2/p}} < 1.
 \end{equation}
 We will determine $\delta > 0$ such that
 \[
 \max\{\dgwz(X,Y),|t-s|\} < \delta \Rightarrow \dgwz(X_s,Y_t) < \epsilon,
 \]
 for an arbitrary $Z$-network $X$ and $s \in [0,1]$. We have
 \begin{align}
     \dgwz(X_s,Y_t) &\leq \dgwz(X_s,X_t) + \dgwz(X_t,Y_t) \\
     &\leq \frac{(3|t-s|)^{1/p}}{2} \mathrm{size}_{p,z}(X) + (1-t)^{2/p} \dgwz(X,Y) \label{eqn:joint_continuity_1} \\
     &\leq  \frac{(3|t-s|)^{1/p}}{2} \mathrm{size}_{p,z}(Y) + (3|t-s|)^{1/p} \dgwz(X,Y) + (1-t)^{2/p} \dgwz(X,Y) \label{eqn:joint_continuity_2}
 \end{align}
 where \eqref{eqn:joint_continuity_1} follows by \Cref{lem:one_point_Holder} and the Lipschitz bound \eqref{eqn:contractible_proof_3}, and \eqref{eqn:joint_continuity_2} follows by the lower bound \eqref{eqn:size_bound} on $Z$-GW distance in terms of size functions. Choose $\delta > 0$ such that
 \[
 \delta < \min\left\{\frac{1}{3}\left(\frac{2\epsilon}{3\mathrm{size}_{p,z}(Y)}\right)^p, \frac{\epsilon}{3(1-t)^{2/p}}   \right\}.
 \]
 The assumption
 $
 \max\{\dgwz(X,Y),|t-s|\} < \delta
 $
 then implies 
 \begin{align}
     &\frac{(3|t-s|)^{1/p}}{2} \mathrm{size}_{p,z}(Y) + (3|t-s|)^{1/p} \dgwz(X,Y) + (1-t)^{2/p} \dgwz(X,Y) \\
     &< \frac{2\epsilon}{3\mathrm{size}_{p,z}(Y)} \cdot \frac{1}{2} \mathrm{size}_{p,z}(Y) + \frac{2\epsilon}{3\mathrm{size}_{p,z}(Y)} \cdot \frac{\epsilon}{3(1-t)^{2/p}} + (1-t)^{2/p} \cdot \frac{\epsilon}{3(1-t)^{2/p}} \\
     &= \frac{\epsilon}{3} + \frac{2\epsilon}{3\mathrm{size}_{p,z}(Y)(1-t)^{2/p}} \cdot \frac{\epsilon}{3} + \frac{\epsilon}{3} < \epsilon, \label{eqn:contrability_proof_7}
 \end{align}
 where we have used the assumption \eqref{eqn:epsilon_assumption} to simplify the middle term in \eqref{eqn:contrability_proof_7}. This proves  continuity of $\Phi$ at any $Z$-network $Y$ with $\mathrm{size}_{p,z}(Y) > 0$ and $t < 1$. 
 
  It remains to consider the cases $\mathrm{size}_{p,z}(Y) = 0$ or $t = 1$. In either case, $Y_t$ is weakly isomorphic to $X_1$, and continuity amounts to controlling $\dgwz(X_s,X_1)$. Applying \Cref{lem:one_point_Holder} and the size bound  \eqref{eqn:size_bound} gives
  \begin{align}
  \dgwz(X_s,Y_t) = \dgwz(X_s,X_1) &\leq \frac{(3(1-s))^{1/p}}{2} \mathrm{size}_{p,z}(X)\\
  &\leq \frac{(3(1-s))^{1/p}}{2}(\mathrm{size}_{p,z}(Y) + 2 \dgwz(X,Y)). \label{eqn:special_case_bound}
  \end{align}
  If $t=1$, then \eqref{eqn:special_case_bound} can be made arbitrarily small by taking $s$ sufficiently close to $1$ and $X$ sufficiently close to $Y$. If $\mathrm{size}_{p,z}(Y)= 0$, then we get 
  \[
   \frac{(3(1-s))^{1/p}}{2}(\mathrm{size}_{p,z}(Y) + 2 \dgwz(X,Y)) \leq 3^{1/p} \dgwz(X,Y),
  \]
  which gives control of $\dgwz(X_s,Y_t)$ in terms of $\dgwz(X,Y)$, and this completes the proof. 
\end{proof}

In analogy with the discussion in the previous subsection, the proof idea for this result fails in the $p=\infty$ case, which remains open. \begin{question}\label{question:contractibility}
    How does contractibility of $\mathcal{M}_\sim^{Z,\infty}$ depend on that of $Z$?
\end{question}

\subsubsection{Geodesicity}

A natural question is whether the path-connectedness result \Cref{prop:path-connected} can be pushed further to a statement about the existence of geodesics. Although this question is not fully solved, we have the following result passing the geodesicity of $Z$ to $\MZp$.

\begin{theorem}\label{thm:geodesic}
    If $Z$ is geodesic, then so is $\MZp$.
\end{theorem}
\begin{proof}
    Suppose that $(Z,d_Z)$ is geodesic. Let $X_i = (X_i,\omega_i,\mu_i) \in \M^{Z,p}$ for $i \in \{0,1\}$ and let $\pi$ be an optimal coupling realizing $\dgwz(X_0,X_1)$. For $t \in [0,1]$, let $X_t \in \M^{Z,p}$ be the $Z$-network
    \[
        X_t = (X_0 \times X_1, \omega_t, \pi),
    \]
    with $\omega_t:(X_0 \times X_1) \times (X_0 \times X_1) \to Z$ defined as follows. For $(x_0,x_1),(x_0',x_1') \in X_0 \times X_1$, choose
    \[
        \omega_\bullet ((x_0,x_1),(x_0',x_1')): [0,1] \to Z: t \mapsto \omega_t ((x_0,x_1),(x_0',x_1'))
    \]
    to be a geodesic joining $\omega_0(x_0,x_0')$ to $\omega_1(x_1,x_1')$ in $Z$; in particular, $\omega_i((x_0,x_1),(x_0',x_1')) = \omega_i(x_i,x_i')$ for $i \in \{0,1\}$. It is enough to prove that $\dgwz(X_s,X_t)\leq |s-t|\dgwz(X_0,X_1)$ for any $s,t\in [0,1]$.

    Let $X = X_0\times X_1$ and construct a coupling of $\pi$ and $\pi$ on $X\times X$ by $\Delta_* \pi $ where $\Delta: X\to X\times X$ is the standard diagonal map, $\Delta(x)=(x,x)$. Then, for any $s,t\in [0,1]$, and for $p \in [1,\infty)$,
    \begin{align}
        &2^p\dgwz(X_s, X_t)^p \\
        &\quad \leq \int_{X\times X}\int_{X\times X}d_Z(\omega_s(x,x'), \omega_t(y,y'))^p\Delta_*\pi(dx\times dy)\Delta_*\pi(dx'\times dy')                      \\ &\quad= \int_{X}\int_{X}d_Z(\omega_s(x,x'), \omega_t(x,x'))^p\pi(dx) \pi(dx') \label{eqn:geodesic_proof_1} \\
                          &\quad = |s-t|^p\int_{X}\int_{X}d_Z(\omega_0(x,x'),\omega_1(x,x'))^p \pi(dx)\pi(dx')   \label{eqn:geodesic_proof_2}                                                        \\
                          &\quad =|s-t|^p\int_{X_0\times X_1}\int_{X_0\times X_1}d_Z(\omega_0(x_0,x_0'),\omega_1(x_1,x_1'))^p \pi(dx_0\times dx_1)\pi(dx_0'\times dx_1') \label{eqn:geodesic_proof_3} \\
                          &\quad = |s-t|^p 2^p\dgwz(X_0,X_1)^p. \label{eqn:geodesic_proof_4}
    \end{align}
    The argument uses the change of variable formula in \eqref{eqn:geodesic_proof_1}, that $\omega_t$ is a geodesic in \eqref{eqn:geodesic_proof_2}, the definition of $\omega_0$ and $\omega_1$ in \eqref{eqn:geodesic_proof_3} and that $\pi$ is an optimal coupling in \eqref{eqn:geodesic_proof_4}. The $p=\infty$ case follows from the observation that the proof can be rewritten in terms of $L^p$ norms, in which case it applies directly.
\end{proof}

It is currently an open question if we have the geodesic property of $\MZp$ when $Z$ is not geodesic. Although not a counterexample, the following example partially explains how the Gromov-Wasserstein distance for a discrete $Z$ behaves.
\begin{example}
    Fix $p=1$. Let $Z=\{0,1\}$ with a discrete metric and consider $Z$-networks $X_0=([0,1],\omega_0,\mathscr{L})$ and $X_1= ([0,1],\omega_1,\mathscr{L})$ where $\omega_0$ and $\omega_1$ are constant functions returning the values $0$ and $1$, respectively. The path $X_t = ([0,1],\omega_t,\mathscr{L})$ where 
    \begin{equation}
        \omega_{t}(u,v) = \begin{cases}
            0, & v\leq t \\ 1, & \textrm{otherwise}
        \end{cases}
    \end{equation}
    defines a geodesic between $X_0$ and $X_1$ because, for $0\leq s\leq t\leq 1$, by \Cref{lem: lp comparison}, we have $\mathrm{GW}_1^Z(X_s,X_t) \leq \frac{1}{2} D_1(\omega_s,\omega_t)$, and
    \begin{align}
        D_1(\omega_s,\omega_t) =\int_{0}^{1}\int_{0}^{1}d_Z(\omega_0(u,v),\omega_1(u,v))dudv =\int_{s}^{t}\int_{0}^{1}dudv = t-s
    \end{align}
    Now, since $\omega_0,\omega_1$ are constant, $X_0$ and $X_1$ are weakly isomorphic to $(\{0\},\bar{\omega}_0,\delta_{0})$ and $(\{1\},\bar{\omega}_1,\delta_{1})$, respectively, where $\bar{\omega}_0,\bar{\omega}_1$ here are again the constant function with values $0,1$, respectively. As the product measure is the only coupling between Diracs, we can explicitly calculate $\mathrm{GW}_1^Z(X_0, X_1)=\frac{1}{2}$. Combining together with the previous result, we have $\mathrm{GW}_1^Z(X_s,X_t)\leq (t-s)\mathrm{GW}_1^Z(X_0,X_1)$, which is sufficient to show that $X_t$ is a geodesic.
\end{example}
As we can see from this example, although $Z=\{0,1\}$ is discrete, we can find a geodesic connecting between $X_0$ and $X_1$. Moreover, it is crucial that $p=1$ because of the following.
\begin{proposition}
Let $Z=\{0,1\}$ endowed with the discrete metric. Then, $\MZp$ is not geodesic for $p>1$.   
\end{proposition}
\begin{proof}
    Consider $Z$-networks $X=(\{0\},\omega_0, \delta_0), Y= (\{0\}, \omega_1, \delta_0)$ where $\omega_0, \omega_1$ are the functions with $\omega_0(0,0)=0$ and $\omega_1(0,0)=1$.  Suppose that $\MZp$ is geodesic. Then, there is a \emph{midpoint} between them~\citep[Lemma 2.4.8]{burago2001course}. That is, there exists a $Z$-network $M= (M,\omega_M, \mu_M)$ such that
    \begin{equation}
        \dgwz(X,M)=\dgwz(Y,M)=\frac{1}{2}\dgwz(X,Y).
    \end{equation}
We will now calculate each term in the equation. Since the only coupling between $\delta_0$ and $\mu_M$ is the product measure $\delta_0\otimes \mu_M$, it is automatically the optimal coupling, and we can calculate $\dgwz(X,M)$ as
\begin{equation}
    \dgwz(X,M) = \frac{1}{2}\left(\int_{M}\int_{M}d_Z(0,\omega_M(m,m'))^p\mu_M(dm)\mu_M(dm')\right)^{1/p}
\end{equation}
Since $d_Z(0,\omega_M(m,m'))$ is nonzero if and only if $\omega_M(m,m')=1$, we have 
\begin{equation}
    \dgwz(X,M) =\frac{1}{2} \mu_M\otimes \mu_M(\omega_M^{-1}(\{1\}))^{1/p}
\end{equation}
Similarly, we have $\dgwz(Y,M)=\frac{1}{2}\mu_M\otimes \mu_M(\omega_M^{-1}(\{0\}))^{1/p}$ and $\dgwz(X,Y)=\frac{1}{2}$. Therefore, we obtain
\begin{equation}
    \frac{1}{2}\mu_M\otimes \mu_M(\omega_M^{-1}(\{1\}))^{1/p}=\frac{1}{2}\mu_M\otimes \mu_M(\omega_M^{-1}(\{0\}))^{1/p} = \frac{1}{4}
\end{equation}
This equation implies $1=\mu_M\otimes \mu_M(M\times M) = \mu_M\otimes \mu_M(\omega_M^{-1}(\{0\})) +\mu_M\otimes \mu_M(\omega_M^{-1}(\{1\})) = \frac{1}{2^{p-1}}$, but this is contradiction since $p>1$.
\end{proof}
These two results suggest that the geodesicity of $\dgwz$ differs between $p=1$ and $p>1$ cases. In fact, the Wasserstein distance has such a property: the Wasserstein $p$-space on any Polish metric space $Z$ is geodesic if $p=1$, but geodesicity depends on the underlying metric if $p>1$ \citep[Theorem 3.16, Remark 3.18]{MEMOLI2023102006}. Considering this result, we conclude this subsection by posing the following questions:

\begin{question}\label{question:always_geodesic}
For $p=1$, is $\MZp$ geodesic for any Polish metric space $Z$?
\end{question}

\begin{question}\label{question:converse_geodesic}
For $p>1$,  does $\MZp$ being geodesic imply that $Z$ is geodesic?
\end{question}

\section{Lower Bounds and Computational Aspects}\label{sec:approximations}

As discussed by  \citet{memoli2011} and \citet{chowdhury2019gromov}, the exact computation of the Gromov-Wasserstein distance is equivalent to solving non-convex quadratic programming, which is NP-hard. However, we can still obtain computationally tractable lower bounds on the distance through \textbf{invariants} of $Z$-networks. An invariant of a $Z$-network is defined in terms of some map $\phi:\MZp \to (I,d_I)$, where $(I,d_I)$ is a relatively simple pseudometric space, with the property that $\dgwz(X,Y) = 0$ implies $\phi(X) = \phi(Y)$. The associated invariant of $X$ is then $\phi(X)$. Example target spaces $(I,d_I)$ include the real line, the Wasserstein space, or $\MZp$ with simpler pseudometrics defined by optimal transport problems. Through such invariants, we will be able to transform the problem of estimating the GW distance into the problem of calculating the distance in $(I, d_I)$.

In this section, we will generalize previous computational results on invariants of the standard Gromov-Wasserstein distance to our framework. We also provide a new result about approximating the $Z$-GW distance, for arbitrary $Z$, by the $\mathbb{R}^n$-GW distance.

\subsection{Hierarchy of Lower Bounds}
Following the previous works (\citealt[Section 3]{chowdhury2019gromov}; \citealt[Section 6]{memoli2007}; \citealt[Section 6]{memoli2011}), we consider the lower bounds and invariants of the $Z$-network GW distance. We first define the generalizations of invariants given by \citet[Section 3]{chowdhury2019gromov}.

\begin{definition}[$Z$-Network Invariants]\label{def:invariants}
    Let $X=(X,\omega_X,\mu_X)$ and $Y=(Y, \omega_Y,\mu_Y)$ be  $(Z,p)$-networks. We define the following invariants:
    \begin{enumerate}
        \item The \define{size} of $X$ given a base point $z_0\in Z$ is
              \begin{equation}
                \label{def:size}
                  \textnormal{size}_{p,z_0}(X) = \| d_Z(\omega_{X}(\cdot,\cdot),z_0)\|_{L^p(\mu_X\otimes \mu_X)}
              \end{equation}
        \item The \define{outgoing joint eccentricity function} $\eccout : X\times Y \to \mathbb{R}_{+}$ is defined by
              \begin{equation}
                  \eccout(x,y) = \inf_{\pi\in \mathcal{C}(\mu_X,\mu_Y)}\|d_Z(\omega_{X}(x,\cdot),\omega_Y(y,\cdot))\|_{L^p(\pi)}
              \end{equation}
        \item The \define{incoming joint eccentricity function} $\eccin : X\times Y \to \mathbb{R}_{+}$ is defined by
              \begin{equation}
                  \eccin(x,y) = \inf_{\pi\in \mathcal{C}(\mu_X,\mu_Y)}\|d_Z(\omega_{X}(\cdot,x),\omega_Y(\cdot,y))\|_{L^p(\pi)}
              \end{equation}
        \item The \define{outgoing eccentricity} of $X$ given a base point $z_0$ is
              \begin{equation}
                  \eccoutpX(x) = \|d_Z(\omega_{X}(x,\cdot),z_0)\|_{L^p(\mu_X)}
              \end{equation}
        \item The \define{incoming eccentricity} of $X$ given a base point $z_0$ is
              \begin{equation}
                  \eccinpX(x) = \|d_Z(\omega_{X}(\cdot,x),z_0)\|_{L^p(\mu_X)}
              \end{equation}
    \end{enumerate}
\end{definition}
We note that many of the invariants are defined for a base point $z_0\in Z$. For the case $Z=\mathbb{R}$ considered by \citet{chowdhury2019gromov}, the base point was implicitly chosen to be $0$. However, in the general case, there is no canonical choice of the base point. One of the major results by \citet{chowdhury2019gromov} was the hierarchy of lower bounds for the GW distance for $\mathbb{R}$-networks. We generalize this result to the $Z$-network case along with the definitions of the lower bounds.

\begin{theorem}[Hierarchy of lower bounds]
    \label{thm:lower_bounds}
    Let $X=(X,\omega_X,\mu_X), Y=(Y,\omega_Y,\mu_Y)$ be $Z$-networks. We fix $z_0\in Z$ and define a cost function $C:X\times Y\to \mathbb{R}_{+}$ by $C(x,y) = W_p(\omega_X(x,\cdot)_*\mu_X, \omega_Y(y,\cdot)_* \mu_Y)$. Then we have the following, for $p\in [1,\infty]$ and $z_0\in Z$:
    \begin{align}
        \dgwz(X,Y) & \geq \frac{1}{2}\inf_{\pi \in \mathcal{C}(\mu_X,\mu_Y)} \|\eccout\|_{L^p(\pi)} \tag{TLB} \label{ineq:TLB}                                         \\ &=\frac{1}{2}\inf_{\pi \in \mathcal{C}(\mu_X,\mu_Y)}\|C\|_{L^p(\pi)} \tag{$Z$-TLB} \label{ineq: Z-TLB} \\ &\geq \frac{1}{2}\inf_{\pi \in \mathcal{C}(\mu_X,\mu_Y)}\|\eccoutpX -\eccoutpY\|_{L^p(\pi)} \tag{FLB} \label{ineq:FLB} \\
                   & = \frac{1}{2}\mathrm{W}_p((\eccoutpX)_*\mu_X, (\eccoutpY)_*\mu_Y) \tag{$Z$-FLB} \label{ineq: Z-FLB}                                                        \\
                   & \geq \frac{1}{2}|\textnormal{size}_{p,z_0}(X)-\textnormal{size}_{p,z_0}(Y)| \tag{SzLB}    \label{ineq:SzLB}                                       \\
        \dgwz(X,Y) & \geq \frac{1}{2}\inf_{\pi \in \mathcal{C}(\mu_X\otimes \mu_X,\mu_Y\otimes \mu_Y)}\|d_Z(\omega_X,\omega_Y)\|_{L^p(\pi)} \tag{SLB} \label{ineq:SLB} \\ &= \frac{1}{2}\mathrm{W}_p((\omega_X)_*(\mu_X\otimes \mu_X), (\omega_Y)_*(\mu_Y\otimes \mu_Y)) \tag{$Z$-SLB} \label{ineq: Z-SLB}
    \end{align}
    and similar inequalities for the incoming eccentricity functions.
\end{theorem}
We note that the terminology FLB, SLB and TLB stands for ``first", ``second" and ``third" lower bound, respectively---this is in reference to the prior conventions \citep{memoli2011,chowdhury2019gromov}.
To prove this theorem, we follow the same strategy as \citet[Theorem 24]{chowdhury2019gromov}. Therefore we first introduce the following lemma, whose proof is provided in \Cref{app:lemma28}.

\begin{lemma}[{\citealt[Lemma 28]{chowdhury2019gromov}}]
\label{lem:lemma28}
    Let $X,Y,Z$ be Polish, and let $f:X\to Z$ and $g:Y\to Z$ be measurable. Let $T:X\times Y\to Z\times Z$ be the map $(x,y)\mapsto (f(x),g(y))$. Then we have:
    \begin{equation}
        T_* \mathcal{C}(\mu_X,\mu_Y) = \mathcal{C}(f_*\mu_X,g_*\mu_Y)
    \end{equation}
    Consequently,
    \begin{equation}
        \mathrm{W}_p(f_*\mu_X,g_*\mu_Y) = \inf_{\pi\in \mathcal{C}(\mu_X,\mu_Y)}\|d_Z(f,g)\|_{L^p(\pi)}
    \end{equation}
\end{lemma}

\begin{proof}of \Cref{thm:lower_bounds}.
    The inequality \eqref{ineq:TLB} is obtained through
    \begin{align}
        \dgwz(X,Y) &\geq \frac{1}{2}\inf_{\pi, \nu\in \mathcal{C}(\mu_X,\mu_Y)}\|d_Z(\omega_X,\omega_Y)\|_{L^p(\pi \otimes \nu)} \\ 
        &= \frac{1}{2}\inf_{\pi\in  \mathcal{C}(\mu_X,\mu_Y)}\inf_{\nu\in  \mathcal{C}(\mu_X,\mu_Y)}\|\|d_Z(\omega_X(x,\cdot),\omega_Y(y,\cdot))\|_{L^p(\nu)}\|_{L^p(\pi)} \\
        &\geq \frac{1}{2}\inf_{\pi\in  \mathcal{C}(\mu_X,\mu_Y)} \left\| \inf_{\nu\in  \mathcal{C}(\mu_X,\mu_Y)}\|d_Z(\omega_X(x,\cdot),\omega_Y(y,\cdot))\|_{L^p(\nu)}\right\|_{L^p(\pi)}\\
        &= \frac{1}{2}\inf_{\mu\in \mathcal{C}(\mu_X,\mu_Y)}\|\eccout\|_{L^p(\pi)}.
    \end{align}
    The equality \eqref{ineq:TLB} $=$ \eqref{ineq: Z-TLB} is obtained by applying \Cref{lem:lemma28}, by setting $f = \omega_X(x,\cdot), g = \omega_Y(y,\cdot)$. The inequality \eqref{ineq:TLB} $\geq$ \eqref{ineq:FLB} is obtained by applying the reverse triangle inequality 
    \[
    d_Z(\omega_X(x,\cdot), \omega_Y(x,\cdot)) \geq |d_Z(\omega_X(x,\cdot),z_0) - d_Z(\omega_Y(y,\cdot),z_0)|
    \]
    and the reverse Minkowski inequality 
    \[
    \|d_Z(\omega_X(x,\cdot),z_0) - d_Z(\omega_Y(y,\cdot),z_0)\|_{L^p(\pi)} \geq \left| \|d_Z(\omega_X(x,\cdot),z_0)\|_{L^p(\pi)}-\|d_Z(\omega_Y(y,\cdot),z_0)\|_{L^p(\pi)}\right|.
    \]
    The equality \eqref{ineq:FLB} $=$ \eqref{ineq: Z-FLB} is again obtained via \Cref{lem:lemma28} by setting $f = \eccoutpX, g=\eccoutpY$. The inequality \eqref{ineq:FLB} $\geq$ \eqref{ineq:SzLB} is obtained by another application of the reverse Minkowski inequality, 
    \[
    \|\eccoutpX-\eccoutpY\|_{L^p(\pi)} \geq \left|\|\eccoutpX\|_{L^p(\pi)}-\|\eccoutpY\|_{L^p(\pi)}\right|,
    \]
    and the fact that 
    \[
    \|\eccoutpX\|_{L^p(\pi)}=\|\eccoutpX\|_{L^p(\mu)}=\size(X).
    \]
    Finally, the inequality \eqref{ineq:SLB} is obtained by considering an optimal coupling $\pi \in \mathcal{C}(\mu_X,\mu_Y)$ and applying the following chain of inequalities for $\pi\otimes \pi\in \mathcal{C}(\mu_X\otimes \mu_X, \mu_Y\otimes \mu_Y)$:
    \begin{equation}
        \dgwz(X,Y) = \frac{1}{2}\|d_Z(\omega_X,\omega_Y)\|_{L^p(\pi\otimes\pi)} \geq \frac{1}{2}\inf_{\sigma \in \mathcal{C}(\mu_X\otimes \mu_X, \mu_Y\otimes \mu_Y)}\|d_Z(\omega_X,\omega_Y)\|_{L^p(\sigma)}
    \end{equation}
    and the equality \eqref{ineq:SLB} $=$ \eqref{ineq: Z-SLB} is obtained via \Cref{lem:lemma28} by setting $f = \omega_X, g = \omega_Y$.
\end{proof}

\subsection{Approximation of \texorpdfstring{$Z$}{Z}-Network GW Distances via \texorpdfstring{$\mathbb{R}^n$}{Rn}-Networks}
In the previous subsection, we note that \eqref{ineq:FLB} was obtained by the reverse triangle inequality $|d_Z(\omega_X,z_0)-d_Z(\omega_Y,z_0)|\leq d_Z(\omega_X,\omega_Y)$. This inequality can be seen as relating the $Z$-networks $(X, \omega_X, \mu_X)$ and $(Y,\omega_Y, \mu_Y)$ to the $\mathbb{R}$-networks $(X, d(\omega_X,z_0), \mu_X)$ and $(Y, d(\omega_X,z_0), \mu_Y)$, respectively. We will now generalize this idea to $\mathbb{R}^n$-networks.

\begin{theorem}
    \label{thm: approximation_Rn}
    Suppose $Z$ is a bounded metric space. Let $X=(X,\omega_X,\mu_X)$, $Y=(Y,\omega_Y,\mu_Y)$ be $Z$-networks. We fix $Q=(q_1,\ldots,q_n)\in Z^n$  and define $\mathbb{R}^n$-networks by 
    \[
    X_{Q}=(X,\omega_{X_Q},\mu_X) \quad \mbox{and} \quad Y_{Q}=(Y, \omega_{Y_Q},\mu_Y),
    \]
    where 
    \[
    \omega_{X_Q} = (d_Z(\omega_X,q_i))_{i=1}^{n} \quad \mbox{and} \quad  \omega_{Y_Q}=(d_Z(\omega_Y,q_i))_{i=1}^{n}.
    \]
    Then we obtain
    \begin{equation}
        n^{-1/r}\dgwRn(X_{Q},Y_{Q}) \leq \dgwz(X, Y) \leq \dgwRn(X_{Q}, Y_{Q}) + \textnormal{H}(Z,Q),
    \end{equation}
    where $\textnormal{H}(Z,Q)$ is the Hausdorff distance between $Z$ and $Q$, the distance on $\mathbb{R}^n$ is the $\ell^{r}$ distance with $r\in [1,\infty]$, and we set $n^{-1/\infty} = 1$. 
\end{theorem}

\begin{remark}
 Before proving the theorem, we provide a few remarks on the statement and its interpretation. 
 
\smallskip\noindent {\bf Generalization to norms other than $\ell^r$ norms.}  Since all norms on a finite-dimensional vector space are equivalent, \Cref{thm: approximation_Rn} can be generalized to the case where we equip $\mathbb{R}^n$ with a norm $\|\cdot\|$ other than the $\ell^r$ norm. The resulting inequality is as follows:
\begin{equation}
        A\cdot\dgwRn(X_{Q},Y_{Q}) \leq \dgwz(X, Y) \leq B\cdot\dgwRn(X_{Q}, Y_{Q}) + \textnormal{H}(Z,Q),
\end{equation}
Here, $A$ and $B$ are constants such that $A\|\cdot\|\leq \|\cdot\|_{\ell^{\infty}}\leq B\|\cdot\|$.

\smallskip
\noindent {\bf Necessity of compactness for practical approximation.}
Although \Cref{thm: approximation_Rn} is true assuming only the boundedness of $Z$, it is necessary to impose a stronger condition, such as compactness, to make the Hausdorff distance between $Z$ and $Q$  arbitrarily small as $n \to \infty$. 

\smallskip
\noindent {\bf The relation between the theorem and the Fr\'{e}chet-Kuratowski theorem.}
In \Cref{thm: approximation_Rn}, setting $r=\infty$ gives no constant factor to the GW distance, and thus the difference between the $\mathbb{R}^n$-GW distance and the $Z$-GW distance admits the following simple Hausdorff distance bound:
\begin{equation}
\label{eqn:gw_approx_H}
\big|\dgwRn(X_{Q},Y_{Q})-\dgwz(X, Y)\big|\leq \textnormal{H}(Z,Q).
\end{equation} 
The inequality shows that the approximation of $\dgwz(X,Y)$ by $\dgwz(X_Q,Y_Q)$ is as accurate as the approximation of $Z$ by $Q$. In other words, it establishes  \emph{consistency} and \emph{stability} of the approximation. This result is observed specifically in the case $r=\infty$, which can be understood through the Fr\'{e}chet-Kuratowski theorem \citep[Proposition 1.17]{Ostrovskii+2013}. This theorem states that any separable metric space isometrically embeds into $\ell^{\infty}(\mathbb{R})$, the space of bounded sequences on $\mathbb{R}$. The embedding $Z \hookrightarrow \ell^{\infty}(\mathbb{R})$ for bounded $Z$ is provided by the mapping $z\mapsto \{d_Z(z,z_i)\}_{i=1}^{\infty}$ where $\{z_i\}_{i=1}^{\infty}$ is any dense subset of $Z$. Consequently, $\omega_{X_Q}=\{d_Z(\omega_X,q_i)\}_{i=1}^{n}$ serves as an approximation of $\omega_X$, identified as $\{d_Z(\omega_X,z_i)\}_{i=1}^{\infty}$, which parallels the Hausdorff approximation of $Z$ by $Q$. On the other hand, for $r<\infty$, the mapping $z\mapsto \{d_Z(z,z_i)\}_{i=1}^{\infty}$ is not necessarily an isometric embedding. As a result, the approximation of $\omega_X$ by $\omega_{X_Q}$ no longer holds.

\smallskip
\noindent {\bf Error bounds for the GW distance approximation.}
Continuing in the $r=\infty$ case, by \eqref{eqn:gw_approx_H}, an error bound for the Hausdorff distance is automatically an error bound for the approximation of the GW distance.
In the compact case, previous works \citep[Theorem 34]{JMLR:v11:carlsson10a}, \citep[Theorem 2]{pmlr-v32-chazal14} show bounds on the convergence rate to 0 of the Hausdorff distance $\textnormal{H}(Z, Q)$ as $n\to\infty$ when $Q=\{q_1,\cdots, q_n\}$ is a random sample from $Z$ (with respect to a measure satisfying certain technical conditions).
\end{remark}

\begin{proof}of \Cref{thm: approximation_Rn}.
    We consider the case $r=\infty$, since other cases follow by using the equivalence of the norm $n^{-1/r}\|x\|_{\ell^r}\leq \|x\|_{\ell^{\infty}} \leq \|x\|_{\ell^r}$. We first prove the lower bound $\dgwRn(X_{Q},Y_{Q}) \leq \dgwz(X, Y)$. Since $|d_Z(\omega_X,q_i)-d_Z(\omega_Y,q_i)|\leq d_Z(\omega_X,\omega_Y)$ for any $i$, we obtain 
    \[
    \|\omega_{X_Q}-\omega_{Y_Q}\|_{\ell^{\infty}} = \max_{i}|d_Z(\omega_X,q_i)-d_Z(\omega_Y,q_i)|\leq d_Z(\omega_X,\omega_Y).
    \]
    Therefore, we obtain the lower bound by integrating both sides against a coupling $\pi\in\mathcal{C}(\mu_X,\mu_Y)$ and taking the infimum in terms of $\pi$. For the upper bound, notice that, for any $z,z'\in Z$, we have  
    \begin{equation}
        d_Z(z,z') \leq \max_{i}|d_Z(z, q_i)-d_Z(z',q_i)| + 2\sup_{z}\min_{i}d_Z(z,q_i)
    \end{equation}
    To verify this, we assume that $d_Z(z,q_i)\geq d_Z(z',q_i)$ without loss of generality because of symmetry. Then, we have
    \begin{equation}
        \label{eqn:ZGWapprox_intermediate}
        d_Z(z,z') \leq |d_Z(z, q_i)-d_Z(z',q_i)| + 2d_Z(z',q_i)
    \end{equation}
    because the right hand side is equal to $d_Z(z,q_i)-d_Z(z',q_i)+2d_Z(z',q_i) = d_Z(z,q_i)+d_Z(z',q_i)$. By the triangle inequality, we obtain the inequality. Thus, taking the maximum in terms of $i$, we have
    \begin{equation}
        d_Z(z,z') \leq \max_{i}|d_Z(z, q_i)-d_Z(z',q_i)| + 2d_Z(z',q_i)
    \end{equation}
    Since the above is true for any $i$, we obtain
    \begin{equation}
        d_Z(z,z') \leq \max_{i}|d_Z(z, q_i)-d_Z(z',q_i)| + 2\min_{i}d_Z(z',q_i).
    \end{equation}
    Taking the supremum in terms of $z'$,
    \begin{equation}
        d_Z(z,z') \leq \max_{i}|d_Z(z, q_i)-d_Z(z',q_i)| + 2\sup_{z'}\min_{i}d_Z(z',q_i),
    \end{equation}
    which is equivalent to \eqref{eqn:ZGWapprox_intermediate}.
    We briefly point out that the boundedness assumption on $Z$ is used here to ensure that $\sup_{z'}\min_{i}d_Z(z',q_i)$  is finite. In addition, the second term $\sup_{z'}\min_{i}d_Z(z',q_i)=\sup_{z}d_Z(z,Q)$ can be bounded by the Hausdorff distance $\textnormal{H}(Z,Q)=\max\{\sup_{z}d_Z(z,Q), \sup_{i}d_Z(q_i,Z)\}$. Now, substituting $z=\omega_X, z'=\omega_Y$, we obtain
    \begin{equation}
        d_Z(\omega_X,\omega_Y) \leq \max_{i}|d_Z(\omega_X, q_i)-d_Z(\omega_Y,q_i)| + 2\textnormal{H}(Z,Q).
    \end{equation}
    For $\pi \in \mathcal{C}(\mu_X,\mu_Y)$, this yields
    \begin{align}
        \|d_Z(\omega_X,\omega_Y)\|_{L^p(\pi\otimes \pi)} & \leq \left\|\max_{i}|d_Z(\omega_X, q_i)-d_Z(\omega_Y,q_i)|+ 2\mathrm{H}(Z,Q)\right\|_{L^p(\pi\otimes \pi)} \\
                                                       & \leq \left\|\max_{i}|d_Z(\omega_X, q_i)-d_Z(\omega_Y,q_i)|\right\|_{L^p(\pi\otimes \pi)} +2\textnormal{H}(Z,Q),
    \end{align}
    where the last inequality follows from Minkowski's inequality and the fact that $2\textnormal{H}(Z,Q)$ is constant.
    Dividing both sides by $2$ and taking the infimum in terms of $\pi$, we obtain the proposition.
\end{proof}

\subsection{Numerical Algorithm}\label{sec:algorithm}
Although the focus of the paper is on theoretical aspects of the $Z$-GW distance, let us briefly detour to sketch a numerical scheme which adapts existing algorithms. 

We start by introducing some notation. Namely, for a 4-dimensional array $(L_{ijkl})_{ijkl}$ and a matrix $(T_{ij})_{ij}$,  we define the product $L\otimes T$ via
\begin{equation}
    L \otimes T = \left(\sum_{kl}L_{ijkl}T_{kl}\right)_{ij}.
\end{equation}
Next, we consider two finite $Z$-networks $(\{1,\cdots, m\},\omega_X,\mu_X)$ and $(\{1,\cdots, n\},\omega_Y,\mu_Y)$. In other words, $\omega_X ,\omega_Y$ are $Z$-valued $m\times m$ and $n\times n$ matrices, respectively, and $\mu_X$, $\mu_Y$ are $m$-dimensional and $n$-dimensional probability vectors,  respectively.
To numerically estimate an optimal coupling $T$---that is, an $\mathbb{R}$-valued $m\times n$ matrix---we perform the following iterations:
\begin{align}
T^{(l+1)} &= \min_{T\in \mathcal{C}(\mu_X,\mu_Y)}
\langle d_Z(\omega_X,\omega_Y)^p\otimes T^{(l)}, T\rangle - \epsilon H(T),\\
\text{where } d_Z(\omega_X,\omega_Y)_{ijkl} &= d_Z\big(\omega_X(i,k), \omega_Y(j,l)\big), \quad
H(T) = -\sum_{i,j} T_{ij}\big(\log T_{ij}-1\big), \quad \epsilon>0 .
\end{align}
Here, for matrices $A,B$ of the same size, $\langle A,B\rangle = \sum_{ij}A_{ij}B_{ij}$. The main difference between applying this algorithm to estimate the $Z$-GW distance as compared to the classical GW distance is the calculation of the product $d_Z(\omega_X,\omega_Y)^p\otimes T^{(l)}$. If $Z=\mathbb{R}$, $d_Z(x,y)=|x-y|$, and $p=2$, the numerical routine of~\cite{peyre2016gromov} reduces the time complexity of the product calculation to $O(n^2m+m^2n)$. In general, the naive calculation of the product costs $O(n^2m^2)$ in time. We also need $O(n^2m^2)$ space to store the 4-dimensional array $d_Z(\omega_X,\omega_Y)$. To alleviate this issue, we can use the approximation by $\mathbb{R}^n$-networks proven in \Cref{thm: approximation_Rn}. In particular, setting $p=r=2$ recovers the situation in \cite{peyre2016gromov} and allows an $O(N(n^2m+m^2n))$ time algorithm where $N$ is the size of $Q$ in Theorem \ref{thm: approximation_Rn}. This is still a cubic-time algorithm, but we believe that a linear-time algorithm is possible by using the Sampled GW algorithm, cf.~\cite{Kerdoncuff2021SampledGW}. We defer addressing these issues and developing a comprehensive numerical framework to future work.

\section{Discussion}\label{sec:discussion}

The $Z$-Gromov-Wasserstein framework defines a very general setting for reasoning about the mathematical properties of the rich and varied GW-like distances which have appeared in the recent literature. These properties have frequently been derived independently in each instance; in contrast, the work here provides a high-level unified perspective, showing that the metric properties of a $Z$-GW distance are closely related to those of the space $Z$. This initial paper on the topic is meant to lay out its foundations, and we envision several directions for future research, both theoretical and applications-oriented in nature. We conclude the paper by outlining some important future goals:

\smallskip
\noindent{\it Geodesic structure and curvature bounds.} The geodesic structure of the $Z$-GW space $\MZp$, and how it relates to that of $Z$, has so far only been  partially characterized in \Cref{thm:geodesic}. Answers to the related open questions, \Cref{question:always_geodesic} and \Cref{question:converse_geodesic}, would provide a more complete characterization.  Moreover, it would be interesting to give conditions under which the geodesics of $Z$-GW distance are always of the form described explicitly in the proof of~\Cref{thm:geodesic}. For example, a result of this form is demonstrated by Sturm for the standard GW distance by~\citet{sturm2023space}, and this is generalized to certain GW variants by~\citet{vayer2020fused,chowdhury2023hypergraph,zhang2024geometry,zhang2024topological}. Such a characterization would be a first step in establishing bounds on the Alexandrov curvature of the $Z$-GW space. Understanding the relationship between the curvature of the space~$Z$ and that of the $Z$-GW space would be very interesting from a theoretical perspective, and could have practical implications to algorithm design. Indeed, a main motivation for the  curvature bounds established by~\citet{sturm2023space} in the standard GW setting was to enable the application of general tools from Alexandrov geometry to study gradient flows of certain functionals on the space of metric measure spaces. On the practical side, curvature estimates can be used to inform algorithms for computing Fr\'{e}chet means of point clouds in a metric space---see, e.g.,~\citet{turner2014frechet,chowdhury2020gromov,zhang2024geometry} for the case of lower bounds, or~\citet{feragen2011means,bacak2014computing} for the case of upper bounds. Additionally, the realization of the $Z$-GW space as a quotient of an $L^p$-space,  described in \Cref{thm:quotient_rep_Lp}, suggests an alternative approach to the theory and numerical modeling of gradient flows in the $Z$-GW space. We plan to explore these ideas in future work.

\smallskip
\noindent{\it Topological Questions.} We showed in \Cref{thm:contractible} that, when $p < \infty$, the $Z$-GW space is contractible, perhaps surprisingly, independently of the topology of $Z$. However, the dependence of the topology of $\mathcal{M}_\sim^{Z,\infty}$ on that of $Z$ has not been resolved---see \Cref{question:path_connectivity} and \Cref{question:contractibility}. We show in \Cref{ex:no_path_continuity} and the ensuing discussion that if $Z$ is not path-connected, then neither is $\mathcal{M}_\sim^{Z,\infty}$, so the answer to these questions is potentially subtle. We remark here that the special case of $p=\infty$ is conceptually important---the distance $\mathrm{GW}_\infty^Z$ is essentially a notion of Gromov-Hausdorff distance for $Z$-valued networks, and Gromov-Hausdorff distance for $\R$-valued networks is already a rich area of study, with ties to areas such as topological data analysis~\citep{chowdhury2018functorial,chowdhury2023distances}. Other interesting topological questions remain open. For example, Gromov's celebrated \emph{precompactness theorem} (see~\citealt[Chapter 7]{burago2022course}) says that a family of metric spaces with certain uniform bounds on geometrical properties of its elements has compact closure in the Gromov-Hausdorff topology, and this is generalized to the Gromov-Wasserstein setting by~\citet[Theorem 5.3]{memoli2011}. A natural question is whether there are interesting precompact families for $Z$-networks in the $Z$-GW topology.

\smallskip
\noindent{\it Structure of Optimal Couplings.} There has recently been significant interest in characterizing the structure of optimal couplings in the GW framework; in particular, several recent articles address the problem of finding general classes of measure networks under which the GW distance is realized by a measure-preserving map (e.g.,~\citealt{memoli2022comparison,delon2022gromov,sturm2023space,dumont2024existence,clark2024generalized}). The $Z$-GW framework gives a more general setting for addressing these questions. The added flexibility in the model might allow for more tractable versions of the problem to be attacked. 

\smallskip
\noindent{\it Computational Pipeline.} The focus of this paper is on the basic theory of $Z$-GW distances, but the results in~\Cref{sec:approximations} point toward more practical considerations. The lower bounds in~\Cref{thm:lower_bounds} are polynomial-time computable in the standard GW setting; the computational efficiency for some of these lower bounds relies on the fact that Wasserstein distances between distributions on the real line can be computed via an explicit formula. From an applications perspective, it would be useful to characterize other spaces $Z$ where the Wasserstein distance is efficiently and explicitly computable (e.g., the circle as done by~\citealt{rabin2011transportation}),  and to utilize this for efficient $Z$-GW distance lower bound computation. Another interesting line of research in this direction is to determine classes of $Z$-networks for which the lower bounds are \emph{injective} in the sense that vanishing of the lower bound implies that the input $Z$-networks are weakly isomorphic---this question is quite subtle in the standard GW setting~\citep{memoli2022distance}. Finally, an important consequence of~\Cref{thm: approximation_Rn} is that, for arbitrary $Z$, $Z$-GW distances can be estimated via solvers for the $\R^n$-GW distance, as we outlined in \Cref{sec:algorithm}. A serious implementation of this approximation scheme will be the subject of a follow-up paper.

\smallskip
\noindent {\it Applications.} The surplus of examples of $Z$-GW distances provided in~\Cref{sec:examples} suggests the wide applicability of this framework. Once the computational pipeline described in the previous paragraph is in place, we plan to apply it to various problems involving analysis of complex and non-standard data. We are particularly interested in exploring the more novel settings described in~\Cref{sec:other_examples}, such as probabilistic metric spaces, shape graphs and connection graphs. 

\acks{We would like to thank Samir Chowdhury for several enlightening discussions about the potential for $Z$-GW distances. T.N. was partially supported by NSF grants DMS--2107808, DMS--2324962 and CIF--2526630. F.M. acknowledges funding from the NSF under grants IIS--1901360, DMS--2301359 and CCF--2310412, and the BSF under grant 2020124. M.B. was partially supported by NSF grants DMS--2324962 and IIS--2426549 and the BSF under grant 2022076.}


\appendix

\section{Proofs of Technical Results}\label{sec:appendix}

In this section, we collect proofs of various technical results whose statements appeared in the main body of the paper. 

\subsection{Proof of \texorpdfstring{\Cref{prop: lp complete_separable}}{Proposition \ref{prop: lp complete_separable}}}\label{app:lp complete_separable}

    \begin{enumerate}[leftmargin=*]
        \item The proof of this part of the proposition can be found in  \citet[Section 1.1]{Korevaar1993SobolevSA}. 
        \item The main idea of the proof is due to \citet{majerMO}. By the theorem of Banach \citep[Chapter XI, $\S$8, Theorem 10]{banach1987theory}, $Y$ can be isometrically embedded onto a subset of a separable Banach space $C([0,1])$ of real-valued continuous functions on $[0,1]$ with the sup-norm. By Pettis measurability theorem \citep[Theorem 1.1.6]{hytonen2016analysis}, all measurable functions with values in $Y$ are strongly measurable by separability of $Y$, so $L^p(X,\mu_X;Y)$ is a subset (or a closed subspace) of the Lebesgue-Bochner space $L^p(X,\mu_X;C([0,1]))$. Since $X$ is separable, its Borel $\sigma$-algebra is countably generated, so $L^p(X,\mu_X;C([0,1]))$ is separable \citep[Theorem 1.2.29]{hytonen2016analysis}. Therefore, $L^p(X,\mu_X;Y)$ is separable.
        \item Let $\epsilon>0$ be given. We first consider the case when $f$ is bounded. By Lusin's theorem, there is a compact set $K\subset X$ such that $\mathscr{L}^d(X\setminus K)<\epsilon$ and $f$ is continuous on $K$. Since $f$ is continuous on a compact set $K$, it is uniformly continuous on $K$. Therefore, there exists $\delta$ such that for any $x,y\in K$, $\|x-y\|_{\mathbb{R}^d}<\delta$ implies $d_Y(f(x),f(y))<\epsilon$. Now, let $h>0$ be small enough so that the maximal distance of each hypercube in the grid is smaller than $\delta>0$. Then, on each cube, take an arbitrary value of $f$ and define a piecewise constant function $g$ that is constant on each cube. Then, for any $x\in K$, $f,g$ satisfies $d_Y(f(x),g(x))<\epsilon$ so that $\sup_{x\in K}d_Y(f(x),g(x))\leq \epsilon$. Now, since $f$ is bounded, there exists $M>0$ such that $d(f(x),f(x'))\leq M$ for any $x,x'\in X$. Therefore,
              \begin{align}
                  D_p(f,g)^p & = \int_{X}d_Y(f(x),g(x))^p\mathscr{L}^d(dx) \\
                  &= \int_{K}d_Y(f(x),g(x))^p\mathscr{L}^d(dx)  + \int_{X\setminus K}d_Y(f(x),g(x))^p\mathscr{L}^d(dx) \\ &\leq \epsilon^p\mathscr{L}^d (K) + M^p\mathscr{L}^d (X\setminus K) \leq \epsilon^p + M^p\epsilon
              \end{align}
              Since $\epsilon$ is arbitrary, we have the proposition for the bounded $f$. We will now consider unbounded $f$. Fix $y_0\in Z$ and define a function $f_n$ by
              \begin{equation}
                  f_n(x) = \begin{cases}
                      f(x) & \textnormal{if } d_Y(f(x),z_0)\leq n \\
                      z_0  & \textnormal{otherwise}
                  \end{cases}
              \end{equation}
              Then, $f_n(x)\to f(x)$ as $n\to \infty$ for any $x\in [0,1]^d$. By definition, we have $d_Y(f_n(x),z_0) \leq d_Y(f(x),z_0)$ so that $d_Y(f_n(x),f(x))^p \leq 2^p d_Y(f(x),z_0)^p$ by triangle inequality, and the $L^p$ assumption on $f$ allows us to dominate the function $d_Y(f_n(x),f(x))^p$ by an integrable function $2^p d_Y(f(x), z_0)$. Therefore, by Lebesgue's dominated convergence theorem, we have $D_p(f_n,f)\to 0$ as $n\to \infty$. Thus, taking a bounded function $f_n$ such that $D_p(f_n,f)<\epsilon/2$ and a piecewise constant function $g$ such that $D_p(f_n,g)<\epsilon/2$, we have $D_p(f,g)<\epsilon$.
    \end{enumerate}
    \hfill $\blacksquare$

\subsection{Proof of \texorpdfstring{\Cref{lem: santam1.8general}}{Lemma \ref{lem: santam1.8general}}}\label{app: santam1.8general}

    We start from the case when $c$ is bounded. Suppose $0\leq c\leq M$ for some $M\geq 0$. We note that $X\times X$ and $Y\times Y$ are again Polish and applying Lusin's theorem to $a$ and $b$, we observe that there are compact sets $K_X\subset X\times X$ and $K_Y\subset Y\times Y$ such that $\mu\otimes \mu(X\times X \setminus K_X)<\delta$ and $\nu\otimes \nu(Y\times Y \setminus K_Y)<\delta$. Now, consider the homeomorphism $s:X\times X\times Y\times Y\ni (x,x',y,y') \mapsto (x,y,x',y')\in X\times Y \times X\times Y$ and the set $K = s(K_X\times K_Y)$. By using $s$, we see that $c(a,b)$, seen as a function on $X\times Y\times X\times Y$, is continuous on $K$. Using that $c$ satisfies $c\leq M$,
    \begin{align}
        \int_{X\times Y}\int_{X\times Y}c(a,b)d(\gamma_n\otimes \gamma_n) & \leq  \int_{(X\times Y)^2}1_{K}\cdot c(a,b)d(\gamma_n\otimes \gamma_n) + \int_{(X\times Y)^2}1_{K^c}\cdot Md(\gamma_n\otimes \gamma_n) \\  &=\int_{(X\times Y)^2}1_{K}\cdot c(a,b)d(\gamma_n\otimes \gamma_n) +M \gamma_n\otimes \gamma_n(K^c)
    \end{align}
    By definition of $K$ and $s$, we have
    \begin{align}
        \gamma_n\otimes \gamma_n(K^c) & = \gamma_n\otimes \gamma_n(s((K_X\times K_Y)^c)) \\ &=\gamma_n\otimes \gamma_n(s([K_X^c\times (Y\times Y)]\cup [(X\times X)\times K_Y^c] )) \\ &= \gamma_n\otimes \gamma_n (s(K_X^c\times (Y\times Y))\cup s((X\times X) \times K_Y^c)) \\ &\leq \gamma_n\otimes \gamma_n(s(K_X^c\times (Y\times Y))) + \gamma_n\otimes \gamma_n (s((X\times X)\times K_Y^c))
    \end{align}
    We note that $(x,y,x',y')\in s(K_X^c\times (Y\times Y))$ if and only if $(x,x')\in K_X^c$, so we have $1_{s(K_X^c\times (Y\times Y))}=1_{K_X^c}(x,x')$ for any $y,y'\in Y$. Therefore,
    \begin{align}
        \gamma_n\otimes \gamma_n(s(K_X^c\times (Y\times Y))) & = \int_{X\times Y}\int_{X\times Y} 1_{K_X^c}(x,x')d(\gamma_n\otimes \gamma_n) \\
                                                             & = \int_{X}\int_{X} 1_{K_X^c}(x,x')d(\mu\otimes \mu)                            \\
                                                             & = \mu\otimes \mu(K_X^c) < \delta
    \end{align}
    We can argue similarly for $K_Y^c$. Thus,
    \begin{align}
        \int_{X\times Y}\int_{X\times Y}c(a,b)d(\gamma_n\otimes \gamma_n) \leq \int_{(X\times Y)^2}1_{K}\cdot c(a,b)d(\gamma_n\otimes \gamma_n) + 2M\delta
    \end{align}
    We note that $1_{K}\cdot c$ is an upper semicontinuous function on $X\times Y$ because it is a nonnegative product of two upper semicontinuous functions. $1_K$ is upper semicontinuous because $K$ is closed. Now recall that the weak convergence $m_k\to m$ of probability measures is equivalent to $\int fm \geq \limsup \int f m_k$ for any upper semicontinuous function $f$ bounded from above. As $1_K c$ satisfies this constraint, using the fact that $\gamma_n \otimes \gamma_n \to \gamma \otimes \gamma $ weakly, we have
    \begin{align}
        \limsup_{n}\int_{X\times Y}\int_{X\times Y}c(a,b)d(\gamma_n\otimes \gamma_n) & \leq  \limsup_{n}\int_{X\times Y}\int_{X\times Y}1_{K}c(a,b)d(\gamma_n\otimes \gamma_n) +2M\delta \\ &\leq \int_{X\times Y}\int_{X\times Y}1_{K}c(a,b)d(\gamma\otimes \gamma) +2M\delta\\ &\leq \int_{X\times Y}\int_{X\times Y}c(a,b)d(\gamma\otimes \gamma) + 2M \delta
    \end{align}
    Since $\delta$ is arbitrary, we have that
    \begin{equation}
        \limsup_{n}\int_{X\times Y}\int_{X\times Y}c(a,b)d(\gamma_n\otimes \gamma_n) \leq \int_{X\times Y}\int_{X\times Y}c(a,b)d(\gamma\otimes \gamma)
    \end{equation}
    implying that the integral functional is upper semicontinuous. We can apply the same argument to $M-c$ to obtain that it is also lower semicontinuous. Overall, we have that the functional is continuous with respect to the weak convergence when $c$ is bounded. Now, we will consider the case when $c$ is unbounded. We can approximate the integral functional by the supremum of continuous functionals by replacing $c$ by monotone sequence $c_n = \min(c,n)$, so the functional is lower semicontinuous. For the upper semicontinuity, the assumption that $c(a,b)\leq f(a)+g(b)$ with $\int (f\circ a)d(\mu\otimes \mu), \int (g\circ b)d(\mu\otimes \mu)<+\infty$ allows us to replace $c$ by $f(a)+g(b)-c(a,b)\geq 0$ and argue the same. Therefore, the continuity is proven for unbounded cases. 
    \hfill $\blacksquare$

\subsection{Proof of \texorpdfstring{\Cref{lem:one_point_Holder}}{Lemma \ref{lem:one_point_Holder}}}\label{app:one_point_Holder}

    Observe that there is a unique coupling $\pi$ between $\mu_X$ and $\delta_\star$, characterized by
    \[
    \pi(dx \times d\star) = \mu_X(dx).
    \]
    The integrals $I_{\pi \pi}$, $I_{X\pi}$ and $I_{\pi Y}$ in \eqref{eqn:proof_of_continuity_paths} can then be computed explicitly by applying the definitions:
    \begin{align}
        I_{\pi \pi} &= \iint_{(X \amalg \{\star\})^4} d_Z(\omega_X \amalg \omega_Y(u,u'), \omega_X \amalg \omega_Y(v,v'))^p \pi(du \times dv) \pi(du' \times dv') \\
        &= \iint_{(X \times \{\star\})^2} d_Z(\omega_X \amalg \omega_Y(x,x'), \omega_X \amalg \omega_Y(\star,\star))^p \pi(dx \times d\star) \pi(dx' \times d\star) \\
        &= \iint_{X \times X} d_Z(\omega_X(x,x'),z)^p \mu_X(dx)\mu_X(dx') = \mathrm{size}_{p,z}(X)^p,
    \end{align}
    \begin{align}
        I_{X\pi} &= \iint_{(X \amalg \{\star\})^4} d_Z(\omega_X \amalg \omega_Y(u,u'), \omega_X \amalg \omega_Y(v,v'))^p \Delta^X_\ast \mu_X(du \times dv) \pi(du' \times dv') \\
        &= \int_{X \times \{\star\}} \int_X d_Z(\omega_X \amalg \omega_Y(x,x'), \omega_X \amalg \omega_Y(x,\star))^p \mu_X(dx) \pi(dx' \times d\star) \\
        &= \iint_{X \times X} d_Z(\omega_X(x,x'), z)^p \mu_X(dx) \mu_X(dx') = \mathrm{size}_{p,z}(X)^p,
    \end{align}
    and, similarly,
    \begin{align}
       I_{\pi Y} &=  \iint_{(X \amalg \{\star\})^4} d_Z(\omega_X \amalg \omega_Y(u,u'), \omega_X \amalg \omega_Y(v,v'))^p \pi(du \times dv) \Delta^Y_\ast \mu_Y(du' \times dv') \\
       &= \int_{X \times \{\star\}} \int_{\{\star\}} d_Z(\omega_X \amalg \omega_Y(x,\star), \omega_X \amalg \omega_Y(\star,\star))^p \pi(dx \times d\star) \delta_\star(d\star) \\
       &= \int_{X \times \{\star\}} \int_{\{\star\}} d_Z(z,z)^p \pi(dx \times d\star) = 0.
    \end{align}
    The estimate \eqref{eqn:proof_of_continuity_paths} therefore becomes
    \[
    2^p \dgwz(X_s,X_t)^p \leq |t-s| \cdot 3 \cdot \mathrm{size}_{p,z}(X)^p,
    \]
    and this can be rewritten as \eqref{eqn:Holder_estimate_one_point}.
    \hfill $\blacksquare$

\subsection{Proof of \texorpdfstring{\Cref{lem:lemma28}}{Lemma \ref{lem:lemma28}}}\label{app:lemma28}

To prove the lemma,  we use the following result regarding \define{analytic sets}, i.e., continuous images of Polish spaces. This result can be found in \citep[Lemma 2.2]{Varadarajan1963Groups} or \citep[Lemma 27]{chowdhury2019gromov}

\begin{lemma}
    \label{lem:varadarajan}
    Let $X, Y$ be analytic subsets of Polish spaces equipped with the relative Borel $\sigma$-fields. Let $f:X\to Y$ be a surjective, Borel-measurable map. Then for any $\nu\in \textnormal{Prob}(Y)$, there exists $\mu\in \textnormal{Prob}(X)$ such that $\nu=f_*\mu$. Here, $\textnormal{Prob}(X)$ is the set of Borel probability measures on $X$.
\end{lemma}

\begin{proof}of \Cref{lem:lemma28}.
    The $\subset$ direction is standard \citep{chowdhury2019gromov}. For the opposite direction, we briefly point out that the proof by \citet{chowdhury2019gromov} contains a gap; to fix it, take a coupling $\pi \in \mathcal{C}(f_*\mu_X, g_*\mu_Y)$, and our goal is to find a measure $\sigma\in \mathcal{C}(\mu_X, \mu_Y)$ such that $T_* \sigma = \pi$. The strategy by \citet{chowdhury2019gromov} was to apply \Cref{lem:varadarajan} directly to the mapping $T$ so that we have a probability measure $\sigma$ such that $\pi=T_*\sigma$, but the issue here is that \Cref{lem:varadarajan} only ensures that $\sigma$ is a probability measure and not necessarily a coupling. Indeed, if $f,g$ are constant maps, $\mathcal{C}(f_*\mu_X, g_*\mu_Y)$ only consists of a Dirac measure, and importantly, any probability measure $\sigma$ (not necessarily a coupling) on $X\times Y$ satisfies $\pi=T_* \sigma$ for the element $\pi\in\mathcal{C}(f_*\mu_X, g_*\mu_Y)$. To impose an additional restriction that $\sigma$ is a coupling, we consider the augmented $T$ mapping $\bar{T}:X\times Y \to X\times Z\times Z\times Y$,
    \begin{equation}
    \bar{T}(x,y)=(\pi_{X}(x,y),T(x,y), \pi_{Y}(x,y)) = (x,f(x),g(y),y)
    \end{equation}
    where $\pi_{X}:X\times Y\to X$ is the projection onto the first component and $\pi_{Y}$ similarly. Intuitively, we added the projection functions $\pi_{X},\pi_{Y}$ so that the condition $\mu_{X}\otimes \pi\otimes \mu_{Y}=\bar{T}_{*}\sigma =(\pi_{X}, T, \pi_{Y})_* \sigma$ we will obtain from \Cref{lem:varadarajan} induces $(\pi_{X})_* \sigma=\mu_X, (\pi_{Y})_* \sigma = \mu_Y$, i.e., the coupling condition.

    Since $X,Y,Z$ are Polish, $X\times Y$ and $X\times Z\times Z\times Y$ are again Polish, and $\bar{T}$ can be made surjective by restricting the range to $\bar{T}(X\times Y)=X\times T(X\times Y)\times Y$. The measurability of $f,g$ implies $\bar{T}$ is measurable, so we apply \Cref{lem:varadarajan} to the measure $\mu_X\otimes \pi|_{T(X\times Y)}\otimes \mu_Y$ where $\pi|_{T(X\times Y)}(A)=\pi(A\cap T(X\times Y))$ and the mapping $\bar{T}$ to obtain a probability measure $\sigma$ on $X\times Y$ such that $\mu_X\otimes \pi|_{T(X\times Y)}\otimes \mu_Y=\bar{T}_* \sigma$. $\sigma$ belongs to $\mathcal{C}(\mu_X,\mu_Y)$ because, by definition, for any measurable $A\subset X$, we have 
    \begin{equation}
    \bar{T}_*\sigma(A\times T(X\times Y)\times Y)=\sigma(A\times Y) = \mu_X(A)\pi|_{T(X\times Y)}(T(X\times Y))\mu_Y(Y)=\mu_X(A)
    \end{equation}
     and similarly for $Y$. Here, the first equality comes from the definition of $\bar{T}^*$ and pushforward, the second is the condition $\mu_X\otimes \pi|_{T(X\times Y)}\otimes \mu_Y=\bar{T}_* \sigma$, and the last one is $\mu_{Y}(Y)=1$ and that $\pi|_{T(X\times Y)}(T(X\times Y)) = \pi(T(X\times Y))=\pi(f(X)\times g(Y))=1$ since $\pi \in \mathcal{C}(f_*\mu_X, g_*\mu_Y)$. 
     
     Now, we have $\pi|_{T(X\times Y)}=T_* \sigma$ because, for any $B\subset T(X\times Y)$, we have $\bar{T}_*\sigma(X\times B\times Y) = T_* \sigma(B) = \pi|_{T(X\times Y)}(B)$. Finally, notice that $\pi(B\cap T(X\times Y)) = \pi(B)$ for any measurable $B\subset Z\times Z$, so $T_*\sigma = \pi|_{T(X\times Y)} = \pi$.
\end{proof}

\vskip 0.2in
\bibliography{bib}

\end{document}